\documentclass[11pt,letterpaper]{amsart}

\usepackage{amssymb}
\usepackage{amsmath}
\usepackage{amsthm}
\usepackage{graphicx}
\usepackage{enumerate}
\usepackage{overpic}
\usepackage{subfigure}
\usepackage{etoolbox}

\usepackage[all]{xy}
\usepackage[dvipsnames]{xcolor}
\usepackage{hyperref}

\theoremstyle{plain}
\newtheorem{theorem}{Theorem}[section]
\newtheorem{lemma}[theorem]{Lemma}
\newtheorem{corollary}[theorem]{Corollary}
\newtheorem{proposition}[theorem]{Proposition}
\newtheorem{conjecture}[theorem]{Conjecture}

\theoremstyle{definition}
\newtheorem{definition}[theorem]{Definition}

\theoremstyle{remark}
\newtheorem{remark}[theorem]{Remark}

\newtheorem*{remark*}{Remark}
\newtheorem*{claim*}{Claim}

\newtheorem{assumption}[theorem]{Assumption}
\newtheorem*{acknowledgments}{Acknowledgments}
\newtheorem*{notation}{Notation}

\numberwithin{figure}{section}

\newcommand{\Int}{\mathrm{int}}

\begin{document}

\title{Taut foliations of 3-manifolds with Heegaard genus two}
\author{Tao Li}
\address{Department of Mathematics \\
 Boston College \\
 Chestnut Hill, MA 02467}
\email{taoli@bc.edu}
\thanks{Partially supported by NSF grant DMS-1906235.}

\begin{abstract}
Let $M$ be a closed, orientable, and irreducible 3-manifold with Heegaard genus two. We prove that if the fundamental group of $M$ is left-orderable then $M$ admits a co-orientable taut foliation. 
\end{abstract}

\maketitle


\section{Introduction}\label{Sintro}

A co-dimension one foliation of a 3-manifold is called a taut foliation if every leaf meets a transverse circle.  Taut foliations have played an important role in the study of 3-manifolds, e.g.~\cite{N, P, G1, G2, G3}. 
The main goal of this paper is to construct taut foliations using information from Heegaard diagrams and fundamental groups. Our main motivation is the $L$-space Conjecture:

\begin{conjecture}[\cite{BGW, J}]
Let $M$ be a closed, orientable, and irreducible 3-manifold. Then the following statements are equivalent.
\begin{enumerate}
	\item $M$ is not a Heegaard Floer $L$-space,
	\item $M$ admits a co-orientable taut foliation,
	\item $\pi_1(M)$ is left-orderable.
\end{enumerate}
\end{conjecture}

The $L$-space conjecture is known to hold for many 3-manifolds, e.g.~graph manifolds \cite{BC, HRRW} and manifolds with positive first Betti number \cite{BRW, G1}. 
It follows from the work of Ozsv\'{a}th-Szab\'{o} \cite{OZ}, Bowden \cite{B}, and Kazez-Roberts \cite{KR} that $M$ cannot be an $L$-space if it admits a taut foliation. Moreover, 
Calegari and Dunfield \cite{CD} have shown that, if $M$ is atoroidal and admits a taut foliation, then $M$ has a finite cover with a left-orderable fundamental group.

Suppose $M$ is a closed, orientable, and irreducible 3-manifold with Heegaard genus two and suppose $\pi_1(M)$ is left-orderable. 
Sarah Rasmussen recently showed that if a minimal intersection  Heegaard diagram is of genus 2 and its associated fundamental group presentation has no subwords that are trivial in $\pi_1(M)$, then $M$ admits a taut foliation, see \cite{R} for an extended abstract in Oberwolfach Reports. 
The main goal of this paper is to prove that $M$ admits a taut foliation without the condition on the Heegaard diagram:

\begin{theorem}\label{Tmain}
Let $M$ be a closed, orientable, and irreducible 3-manifold with Heegaard genus 2.  If $\pi_1(M)$ is left-orderable, then $M$ admits a co-orientable taut foliation.
\end{theorem}

Our proof is very different from Rasmussen's construction. The main tools are branched surfaces. Below is a brief outline:
 
Given a genus-two Heegaard splitting of $M$, consider the 2-complex formed by the Heegaard surface and complete sets of non-separating compressing disks in the two handlebodies. 
This 2-complex can be easily deformed into a branched surface $B$. 
Each branch sector of $B$ is associated with an element in $\pi_1(M)$. We can deform the 2-complex and choose the group elements in such a way that no group element in this construction is negative with respect to a fixed left order on $\pi_1(M)$. By deleting branch sectors whose associated group elements are trivial, we obtain a new branched surface $B_0$. 
The left order of $\pi_1(M)$ gives a splitting instruction so that one can split the branched surface $B_0$ indefinitely. 
The inverse limit of the splitting process is a lamination fully carried by $B_0$. 
A large part of the proof is to show that this lamination can be extended to a taut foliation.

Part of the proof can be easily adapted for certain higher-genus Heegaard diagrams. It is conceivable that there may be a way of modifying Heegaard diagrams so that this construction works in general.

\begin{acknowledgments}
This work was inspired by a lecture given by Sarah Rasmussen at Oberwolfach in February 2020. I would like to thank Sarah  for helpful conversations and thank Mathematisches Forschungsinstitut Oberwolfach for hospitality.
\end{acknowledgments}

\section{Left-orderable groups, branched surfaces, and Heegaard diagrams}\label{Ssetup}

In this section, we review some background on left-orderable groups, branched surfaces, and Heegaard splittings.

\begin{notation}
Throughout this paper, we use $I$ to denote the closed interval $[0,1]$. For any topological space $X$, we use $\Int(X)$, $|X|$, and $\overline{X}$ to denote the interior, number of components, and closure of $X$ respectively.
\end{notation}

\vspace{10pt}
\noindent
\textbf{Left-orderable groups}
\vspace{10pt}

\begin{definition}
A non-trivial group $G$ is \emph{left-orderable} if there exists a strict left-invariant total ordering of elements in $G$, that is, if $f<g$ then $hf<hg$ for any elements $f, g, h \in G$. 
\end{definition}

Left-orderable groups have many interesting properties. For example, every non-trivial element in a left-orderable group has infinite order. 
Many 3-manifold groups are left-orderable. Boyer, Rolfsen, and Wiest \cite{BRW} have shown that a compact irreducible and orientable 3-manifold $N$ has a left-orderable fundamental group if and only if $\pi_1(N)$ surjects onto a left-orderable group. In particular, $\pi_1(N)$ is left-orderable if $H_1(N;\mathbb{Q})\ne 0$. 

Given a left-orderable group $G$ and a left order $<$,
an element $g\in G$ is said to be positive (resp.~negative) if $g>1$ (resp. $g<1$). 
We also use the notation $\le $, that it, $f\le g$ means that either $f<g$ or $f=g$.

\begin{assumption}\label{Assumption}
Throughout the paper, $M$ is a closed, orientable, and irreducible 3-manifold with Heegaard genus 2.  Suppose $\pi_1(M)$ is left-orderable and fix a left order.  
Since Theorem~\ref{Tmain} is known to be true if $H_1(M;\mathbb{Q})\ne 0$ \cite{BRW, G1}, we assume $H_1(M;\mathbb{Q})= 0$.
\end{assumption}

\vspace{8pt}
\noindent
\textbf{Branched surfaces}
\vspace{8pt}

A \emph{branched surface} in $M$ is a union of a finite collection of smooth surfaces glued together forming a compact subspace of $M$ locally modeled on Figure~\ref{Fbranch}(a).

Given a branched surface $B$ in $M$, we use $N(B)$ to denote a regular neighborhood of $B$ and regard $N(B)$ as an $I$-bundle over $B$, see Figure~\ref{Fbranch}(b).

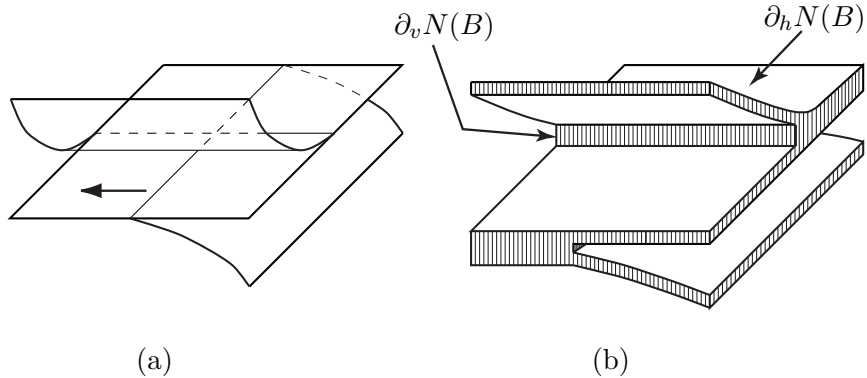
\begin{figure}[h]
	\vskip 0.3cm
	\begin{overpic}[width=4.5in]{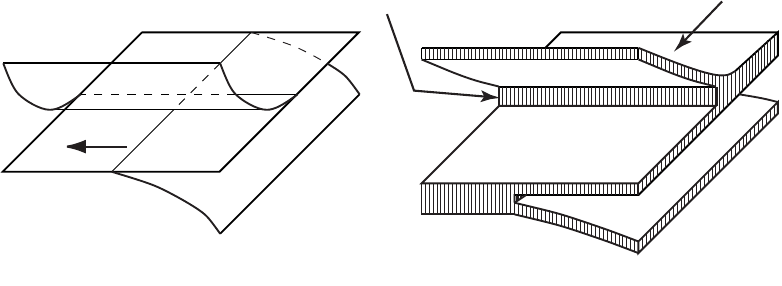}
		\put(15,-3){(a)}
		\put(68,-3){(b)}
		\put(45,36){$\partial_vN(B)$}
		\put(88,37){$\partial_hN(B)$}
	\end{overpic}
	\vskip 0.3cm
	\caption{Branched surface and its fibered neighborhood}\label{Fbranch}
\end{figure}

As illustrated in Figure~\ref{Fbranch}(b),
we divide the boundary of $N(B)$ into two parts: the horizontal boundary $\partial_hN(B)$ and the vertical boundary $\partial_vN(B)$, where $\partial_vN(B)$ consists of non-trivial sub-intervals of the $I$-fibers of $N(B)$ and $\partial_hN(B)$ is transverse to the $I$-fibers.  
In particular, $\partial_vN(B)$ is a collection of annuli. 

Let $\pi\colon N(B)\to B$ be the projection that collapses each $I$-fiber to a point.
The projection $L=\pi(\partial_vN(B))$ is the union of points that do not have neighborhoods (in $B$) homeomorphic to $\mathbb{R}^2$, i.e.~the non-manifold points in $B$. 
$L$ is called the \emph{branch locus} of $B$. 
The projection of a component of $\partial_vN(B)$ is called a component of the branch locus. 
The branch locus is a collection of immersed curves in $B$. 
The direction of the cusp at each arc in $L$ is called the \emph{branch direction} at the arc, see the arrow in Figure~\ref{Fbranch}(a). 
For any component of $B\setminus L$, its closure under the path metric is called a \emph{branch sector} of $B$.

A surface or lamination $\lambda$ is \emph{carried} by $B$ (or $N(B)$) if $\lambda\subset N(B)$ and $\lambda$ meets the $I$-fibers transversely, and $\lambda$ is \emph{fully carried} by $B$ if $\lambda$ meets every $I$-fiber of $N(B)$. 
For any branch sector $s$ of $B$ and any surface $F\subset N(B)$ carried by $B$, we say that $F$ \textit{passes through} the branch sector $s$ if $F$ meets the $I$-fibers of $\pi^{-1}(s)$, where $\pi\colon N(B)\to B$ is the projection.

A \emph{disk of contact} is an embedded disk $D$ in $N(B)$ transverse to the $I$-fibers and with $\partial D\subset \partial_vN(B)$, see Figure~\ref{Fdisk}(a) for a picture in $B$. Note that $D$ may intersect an $I$-fiber of $N(B)$ more than once. We call $D$ a disk of contact with respect to a lamination $\lambda\subset N(B)$ if, in addition, $D\cap\lambda=\emptyset$.  Let $A$ be the component of $\partial_vN(B)$ that contains $\partial D$. We say that $A$ bounds a disk of contact in $N(B)$. 

\begin{figure}[h]
	\begin{overpic}[width=3in]{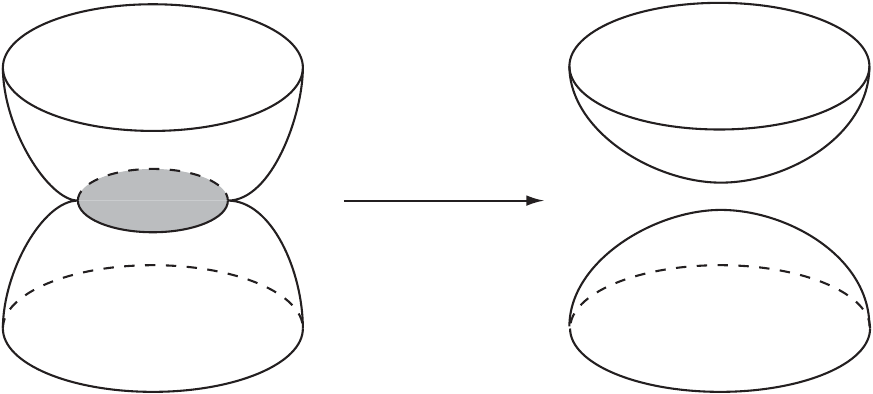}
		\put(16,-7){(a)}
		\put(80,-7){(b)}
		\put(40,24){splitting}
	\end{overpic}
	\vskip 0.4cm
	\caption{Splitting along a disk of contact}\label{Fdisk}
\end{figure}

Let $F\subset N(B)$ be a compact orientable surface transverse to the $I$-fibers of $N(B)$, and let $F\times I$ be a product neighborhood of $F$ with each $\{x\}\times I$ ($x\in F$) a subarc of an $I$-fiber of $N(B)$. 
If we remove $F\times I$, after a small isotopy, $N(B)\setminus\Int(F\times I)$ is a fibered neighborhood $N(B')$ of another branched surface $B'$. We say that $B'$ is obtained by splitting $B$ along $F$ and call $F$ a splitting surface.

If $F$ is a small disk containing a short arc $\alpha$ and $\alpha$ has one or both endpoints in $\partial_vN(B)$, then the splitting along $F$ is illustrated in the 1-dimensional schematic pictures in Figure~\ref{Fsplit1}, where the 3 possible splittings correspond to the 3 possible positions of the arc $\alpha$. 
In fact, any splitting is the result of a sequence of local splittings as illustrated in Figure~\ref{Fsplit1}, plus the splittings along simple disks of contact as shown in Figure~\ref{Fdisk}, see \cite[section 4]{GO} and \cite[section 2.4]{L8} for details. 
Although we usually perform splitting directly on the branched surface as in the top pictures in Figure~\ref{Fsplit1}, by the description above, if $B'$ is obtained by splitting $B$, we view $N(B')\subset N(B)$. 
See \cite[section 2]{L8} for more discussions of splitting branched surfaces.

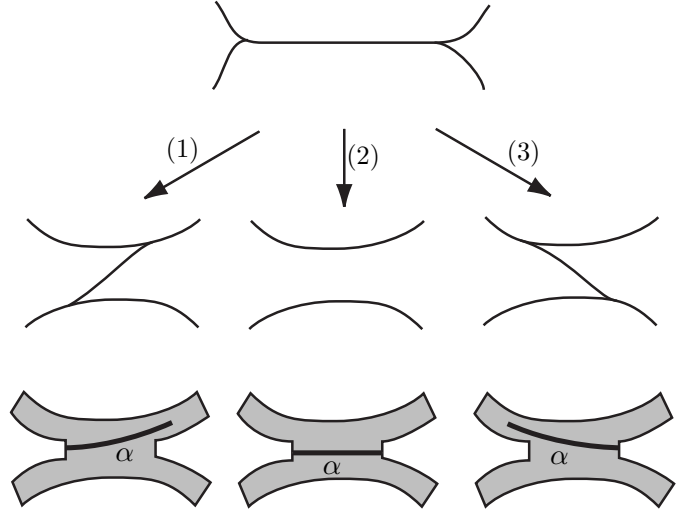
\begin{figure}[h]
	\vskip 0.3cm
	\begin{overpic}[width=3.5in]{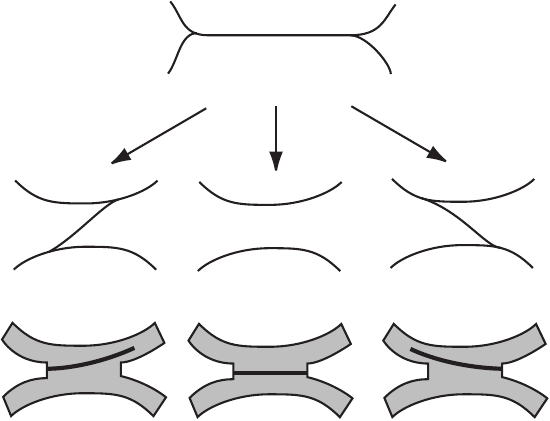}
		\put(16,7.5){$\alpha$}
		\put(47,5.5){$\alpha$}
		\put(81,7){$\alpha$}
		\put(23.5,53){\small{(1)}}
		\put(50.5,52){\small{(2)}}
		\put(74.5,52.5){\small{(3)}}
	\end{overpic}
	\vskip 0.3cm
	\caption{Local splittings}\label{Fsplit1}
\end{figure}

\vspace{8pt}
\noindent
\textbf{Heegaard diagram}
\vspace{8pt}

Let $M=\mathcal{U}\cup_S \mathcal{V}$ be a genus-2 Heegaard splitting along a Heegaard surface $S$. Let $\{U_1, U_2\}$ and $\{V_1, V_2\}$ be complete sets of disjoint non-separating compressing disks in the handlebodies $\mathcal{U}$ and $\mathcal{V}$ respectively, that is, $\mathcal{U}\setminus(U_1\cup U_2)$ and  $\mathcal{V}\setminus(V_1\cup V_2)$ are a pair of 3-balls. 
Let $u_i=\partial U_i$ and $v_i=\partial V_i$ ($i=1,2$) and 
let $\widehat{U}=\{u_1,u_2\}$ and $\widehat{V}=\{v_1,v_2\}$.  The Heegaard surface $S$ together with $\widehat{U}$ and $\widehat{V}$ is called a \textit{Heegaard diagram}.
 
We call the boundary of any non-separating compressing disk in a handlebody a meridian of the handlebody. We say two curves in a surface have tight intersection if their number of intersection points is minimal up to isotopy. 
We may assume the intersection  of $\widehat{U}$ and $\widehat{V}$ is tight.

\begin{lemma}\label{Ldisk}
	Every component of $S\setminus (u_1\cup u_2\cup v_1\cup v_2)$ is a disk.
\end{lemma}
\begin{proof}
	Suppose a component $E$ of $S\setminus (u_1\cup u_2\cup v_1\cup v_2)$ is not a disk. Then there is an essential simple closed curve $\gamma$ in $E$. The curve $\gamma$ must be essential in $S$, as $u_i$ and $v_i$ ($i=1,2$) are essential in $S$. Since $\mathcal{U}\setminus(U_1\cup U_2)$ is a 3-ball, $\gamma$ bounds a compressing disk in $\mathcal{U}$. 
	Similarly, $\gamma$ also bounds a compressing disk in $\mathcal{V}$.  The two compressing disks form a 2-sphere that intersects $S$ in a single curve $\gamma$. 
	Since $M$ is irreducible, this 2-sphere bounds a 3-ball in $M$. 
	This implies that the genus-2 Heegaard splitting is stabilized and not of minimal-genus, contradicting the hypothesis that $M$ has Heegaard genus two.
\end{proof}

A Heegaard diagram can be described using a Whitehead graph: The graph $\Gamma(\widehat{V})$ corresponding to $\widehat{V}$ is obtained by setting the 4 boundary curves of $\overline{S\setminus( v_1\cup v_2)}$ as 4 (fat) vertices $v_1^\pm$ and $v_2^\pm$ and the arcs of $(u_1\cup u_2)\setminus(v_1\cup v_2)$ as edges. Similarly, we also have the Whitehead graph $\Gamma(\widehat{U})$ corresponding to $\widehat{U}$.

A genus-2 Heegaard diagram has many remarkable properties, mainly due to the hyperelliptic involution. 
A theorem of Ochiai \cite{Och} says that the Whitehead graph of any genus-two Heegaard diagram of an irreducible 3-manifold is one of the 3 graphs in Figure~\ref{FHD}. The labels $a, b, c, d$ in Figure~\ref{FHD} denote the numbers of parallel arcs and they may be zero. 
See \cite[section 1.3]{LMP} for more discussions of genus-2 Heegaard diagrams and Whitehead graphs. 

\begin{figure}[h]
	\vskip 0.3cm
	\begin{overpic}[width=4.5in]{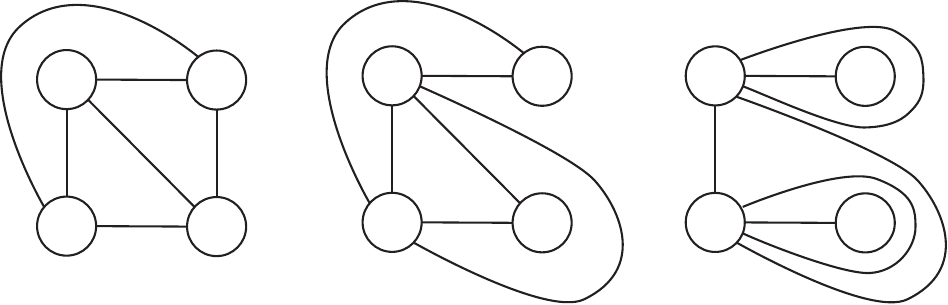}
		\put(5,23){$v_1^+$}
		\put(5,7){$v_1^-$}
		\put(21,23){$v_2^+$}
		\put(21,7){$v_2^-$}
		\put(13,24){$a$}
		\put(13,8.5){$a$}
		\put(13,18){$b$}
		\put(0,29){$b$}
		\put(7.5,15){$c$}
		\put(23,15){$d$}
\put(39.5,23){$v_1^+$}
\put(39.5,7){$v_1^-$}
\put(55,23){$v_2^+$}
\put(55,7){$v_2^-$}
\put(49,24.5){$a$}
\put(49,9){$a$}
\put(50.5,15){$b$}
\put(35,29.5){$b$}
\put(41.5,15){$c$}
\put(60,15.5){$d$}
\put(73,23){$v_1^+$}
\put(73,7){$v_1^-$}
\put(89,23){$v_2^+$}
\put(89,7){$v_2^-$}
\put(84,24.5){$a$}
\put(84,9){$a$}
\put(86,28.5){$b$}
\put(82,12){$b$}
\put(73,15.5){$c$}
\put(96,14){$d$}
		\put(15,-3){\small{(i)}}
		\put(48,-3){\small{(ii)}}
		\put(84,-3){\small{(iii)}}
	\end{overpic}
	\vskip 0.3cm
	\caption{The 3 possible Whitehead graphs}\label{FHD}
\end{figure}
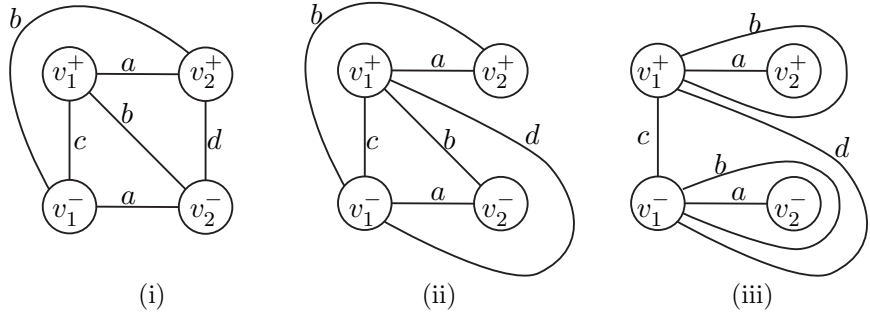

\begin{lemma}\label{LHDparallel}
In $S\setminus (v_1\cup v_2)$, all the subarcs of $u_1$ and $u_2$ connecting $v_1^+$ to $v_2^+$ are parallel.  Similarly, all the subarcs  of $u_1$ and $u_2$ connecting $v_1^-$ to $v_2^-$ are parallel
\end{lemma}
\begin{proof}
The lemma is evident from the 3 possible pictures in Figure~\ref{FHD}, see the edges labeled $a$ in Figure~\ref{FHD}.
\end{proof} 

Given a Heegaard diagram $(S,\widehat{U}, \widehat{V})$, a \emph{wave} with respect to a curve $u_i$ ($i=1,2$) is an arc $\eta$ such that 
\begin{enumerate}
	\item $\eta\cap(\widehat{U}\cup \widehat{V})=\partial\eta\subset u_i$, 
	\item $\eta$ intersects $ u_i$ from the same side of $ u_i$. 
\end{enumerate}
A wave with respect to $v_j$ is defined similarly. 

Let $\eta$ be a wave with respect to $u_1$. One of the three boundary curves of the pair of pants $P_\eta=\overline{N(\eta\cup u_1)}$ is a parallel copy of $u_1$. To simplify notation, we still use $u_1$ to denote this boundary component of $P_\eta$. 
Denote the other two boundary curves of $P_\eta$ by $u'$ and $u''$. 
One of the two curves, say $u'$, is not parallel to $u_i$ ($i=1,2$).  

\begin{claim*}
The curve $u'$ is non-separating in the Heegaard surface $S$.
\end{claim*}
\begin{proof}[Proof of the Claim]
Suppose $u'$ is separating in $S$.  
Hence, $u'$ divides the Heegaard surface into two once-punctured tori, one of which, denoted by $T_1$, contains the pair of pants $P_\eta$.

Next, consider $v_1$ and $v_2$. It follows from the construction of $P_\eta$ that every component of $v_i\cap P_\eta$ ($i=1,2$) is an arc connecting $u_1$ to either $u'$ or $u''$. 
We claim that $u'$ must be disjoint from $v_1\cup v_2$. 
Suppose on the contrary that $v_i\cap u'\ne\emptyset$ and let $\alpha$ be an arc component of $v_i\cap T_1$. 
As $u'=\partial T_1$, $\partial\alpha\subset u'$. 
The two endpoints of $\alpha$ have opposite signs of intersection with $u'$. 
Since the two components of $\alpha\cap P_\eta$ containing $\partial\alpha$ are arcs connecting $u'$ to $u_1$, there are two points of $\alpha\cap u_1$ with opposite signs of intersection. 
However, as $v_i$ and $u_1$ have tight intersection and since $T_1$ is a once-punctured torus, all the intersection points of $\alpha\cap u_1$ must have the same sign of intersection. 
This is a contradiction. Hence $v_i\cap u'=\emptyset$ for both $i=1,2$. 
This means that $u'$ must be disjoint from both $u_1\cup u_2$ and $v_1\cup v_2$, a contradiction to Lemma~\ref{Ldisk}. 
\end{proof}

By the definition of a wave, $u'$ bounds a disk $U'$ in the handlebody $\mathcal{U}$. 
This claim implies that $\{U', U_2\}$ is a new complete set of compressing disks for the handlebody $\mathcal{U}$. 
The \emph{wave move} along $\eta$ is the operation of replacing the disk $U_1$ with the disk $U'$. 
It follows from Lemma~\ref{Ldisk} that the new set of disks $\{U', U_2\}$ has fewer intersection points with $\widehat{V}$. 
See \cite[section 1.3]{LMP} for more discussions of genus-2 Heegaard diagrams and wave moves.

\begin{definition}\label{Dmin}
Given a Heegaard diagram formed by $\{u_1, u_2\}$ and $\{v_1, v_2\}$, we first fix a pair of meridians $v_s$ and $u_t$ ($s=1$ or $2$ and $t=1$ or $2$) such that $|v_s\cap u_t|$ is minimal among all pairs of the meridians from $\{v_1, v_2\}$ and $\{u_1, u_2\}$.
Define the complexity of the Heegaard diagram to be
$c(\widehat{U}, \widehat{V}) = (|v_s\cap u_t|, |(v_1\cup v_2)\cap u_t|, |(v_1\cup v_2)\cap (u_1\cup u_2)|)$ 
with the lexicographical order. We call $v_s$ and $u_t$ the special meridians for the Heegaard diagram. 
\end{definition}

Suppose the Heegaard diagram has minimal complexity in the sense of Definition~\ref{Dmin}. 
This means that the pair of special meridians $v_s$ and $u_t$ in Definition~\ref{Dmin} has minimal intersection number among all pairs of meridians of $\mathcal{V}$ and $\mathcal{U}$.  
Since the wave move reduces the intersection number of pairs of meridians in the Heegaard diagram, a Heegaard diagram with minimal complexity has no wave.

\vspace{8pt}
\noindent
\textbf{Band sum}
\vspace{8pt}

Consider the Heegaard diagram formed by $\{u_1, u_2\}$ and $\{v_1,v_2\}$. 
We cut $S$ open along $u_1$ and $u_2$ and obtain a 4-hole sphere $S^-$. Denote boundary curves of $S^-$ that correspond to the two sides of $u_1$ and $u_2$ by $u_1^\pm$ and $u_2^\pm$ respectively.   
Let $\alpha_1,\dots,\alpha_m$ be the set of subarcs of $v_1\cup v_2$ that connect $u_1^+$ to $u_2^+$ (suppose there is at least one such arc). 
By Lemma~\ref{LHDparallel}, the $\alpha_i$'s are all parallel in $S^-$. 
Let $R$ be a rectangle in $S^-$ containing $\alpha_1,\dots,\alpha_m$, with two opposite boundary edges in $u_1^+$ and $u_2^+$, and disjoint from other subarcs of $v_1$ and $v_2$. 
Let $N(u_1^+\cup R\cup u_2^+)$ be a small neighborhood of $u_1^+\cup R\cup u_2^+$ in $S^-$. 
The closure of $N(u_1^+\cup R\cup u_2^+)$ is a pair of pants, denoted by $P_R$. We view $P_R$ as a sub-surface of $S$. Let $u'$ be the boundary component of $P_R$ that is not $u_1$ or $u_2$. 
We say that $u'$ is a band sum of $u_1$ and $u_2$ along $R$. 

\begin{lemma}\label{Lbandsum} 
	Suppose no subarc of $v_1\cup v_2$ is a wave with respect to $\{u_1, u_2\}$. Let $u'$ be the band sum of $u_1$ and $u_2$ along $R$ described above. Then
	\begin{enumerate}
		\item The curve $u'$ has tight intersection with $v_1$ and $v_2$. 
		\item No subarc of $v_1\cup v_2$ is a wave with respect to $\{u', u_i\}$ ($i=1$ or $2$).
	\end{enumerate}
\end{lemma}
\begin{proof}
First consider the pair of pants $P_R$ in the construction above. Each component of $(v_1\cup v_2)\cap P_R$ is either an arc $\alpha_i$ connecting $u_1^+$ to $u_2^+$ or an arc connecting $u'$ to $u_j^+$ ($j=1,2$). In particular, no arc of $(v_1\cup v_2)\cap P_R$ has both endpoints in the same boundary curve of $P_R$. 

Let $P_R'=S\setminus\Int(P_R)$ be the complement of $P_R$ in $S$. So $P_R'$ is also a pair of pants. 
By the symmetry from the hyperelliptic involution on $S$, there must be a subarc $\xi$ of $v_1\cup v_2$ in $P_R'$ connecting $u_1^-$ to $u_2^-$. 
The arc $\xi$ cuts $P_R'$ into an annulus, which implies that, if a component of $(v_1\cup v_2)\cap P_R'$ has both endpoints in $u'$, then it must be a $\partial$-parallel arc in $P_R'$. 
However, in the construction above, the rectangle $R$ contains all the arcs $\alpha_i$'s connecting $u_1^+$ to $u_2^+$.  Since no subarc of $v_1\cup v_2$ is a wave with respect to $u_1$ and $u_2$, this implies that no component of $(v_1\cup v_2)\cap P_R'$ is a $\partial$-parallel arc with both endpoints in $u'$.  
Moreover, as $\partial P_R'=u'\cup u_1^-\cup u_2^-$, any component of $(v_1\cup v_2)\cap P_R'$ with both endpoints in the same boundary curve $u_i^-$ is a wave with respect to $u_i$, a contradiction to the hypothesis. 
Thus, no component of $(v_1\cup v_2)\cap P_R'$ has both endpoints in the same boundary curve of $P_R'$.

Now consider the set of meridians $\{u', u_1\}$ (the proof for $\{u', u_2\}$ is the same).  Suppose there is a subarc of $v_i$ in $S\setminus(u'\cup u_1)$, denoted by $\gamma$, with both endpoints in $u'$ and connecting $u'$ from the same side. Since there is no such arc totally in $P_R$ or $P_R'$, $\gamma\cap u_2\ne\emptyset$.  
So $u_2$ divides $\gamma$ into a collection of subarcs and each subarc is properly embedded in either $P_R$ or $P_R'$. 
Since $\gamma$ connects $u'$ from the same side and since $\Int(\gamma)\cap (u'\cup u_1)=\emptyset$, $\gamma\cap u_2$ contains more than one point, which implies that a subarc of $\gamma$ has both endpoints in $u_2$ and is properly embedded in either $P_R$ or $P_R'$. 
This contradicts our conclusions above. 
Therefore, there is no such subarc $\gamma$, and 
this means that (1) $u'$ and $v_1\cup v_2$ do not form any bigon and hence have tight intersection and (2) no subarc of $v_1\cup v_2$ is a wave with both endpoints in $u'$. 
Similarly, the argument also implies that no subarc of $v_1\cup v_2$ is a wave with both endpoints in $u_1$. 
\end{proof}

\vspace{8pt}
\noindent
\textbf{Strongly irreducible Heegaard splitting}
\vspace{8pt}

A Heegaard splitting is said to be \textit{reducible} if there is an essential curve on the Heegaard surface that bounds disks in both handlebodies.  
A genus-$g$ ($g\ge 2$) Heegaard splitting $M=\mathcal{U}\cup\mathcal{V}$ is said to be \textit{weakly reducible} if there is a compressing disk $U$ of $\mathcal{U}$ and a compressing disk $V$ of $\mathcal{V}$ such that $\partial U\cap\partial V=\emptyset$. If a Heegaard splitting is not weakly reducible, then it is said to be \textit{strongly irreducible}, see \cite{CG,ST}. 

A minimal-genus Heegaard splitting of an irreducible 3-manifold is always irreducible. 
If a Heegaard splitting is irreducible but weakly reducible, by compressing the Heegaard surface simultaneously on both sides, one can obtain an incompressible surface, see \cite{CG}. 
This implies that if a closed and irreducible 3-manifold $M$ has Heegaard genus two, then any genus-$2$ Heegaard splitting of $M$ is always strongly irreducible. Thus, in this paper, our Heegaard splitting is strongly irreducible.

\section{Construct branched surfaces from a Heegaard diagram}\label{Sbr}

Fix a point $X\in\mathcal{U}\setminus(U_1\cup U_2)$ and a point $Y\in\mathcal{V}\setminus(V_1\cup V_2)$ in the two handlebodies.  
Let $\mathcal{S}=\{s_0, s_1, \dots,s_n\}$ be the set of components of $S\setminus (\partial U_1\cup\partial U_2\cup\partial V_1\cup\partial V_2)$.  By Lemma~\ref{Ldisk}, each $s_i$ is a disk. 
For each disk $s_i$ in $\mathcal{S}$, $i=0,\dots, n$, take a path $l_i$ from $X$ to $Y$ that intersects $s_i$ in a single point. 
We may suppose $l_i$ does not intersect $U_j$, $V_j$, $j=1,2$, or any other disk in $\mathcal{S}$.  
We fix a special disk $s_0$ in $\mathcal{S}$. For any disk $s_i$ in $\mathcal{S}$, let $\gamma_i=l_{0}\cdot l_i^{-1}$ be the loop that starts from the basepoint $X$, first travels along the path $l_{0}$ to $Y$ and then travels along $l_i^{-1}$ back to $X$.  So $\gamma_i$ represents an element in $\pi_1(M, X)$. 
To simplify notation, we also use $\gamma_i$ to denote the element $[\gamma_i]$ in $\pi_1(M, X)$. We call $\gamma_i$ the element of $\pi_1(M)$ associated with the disk $s_i$.  Clearly $\gamma_0=1$ in $\pi_1(M)$.

In this paper, we  are mainly interested in the orders of these elements $\gamma_i$'s with respect to a fixed left order of $\pi_1(M)$. 
However, the definition of $\gamma_i$ depends on the choice of $s_0$. 
Next, we show that the choice of the special disk $s_0$ does not affect the relative orders of these elements. 
Define another set of elements $\eta_i$'s using another disk $s_m\in \mathcal{S}$ instead of $s_0$. More precisely, $\eta_i=l_m\cdot l_i^{-1}$. 
Let $\delta=l_m\cdot l_0^{-1}$. Clearly, $\eta_i=\delta\cdot\gamma_i$. 
Since this is a left order, for any pair of disks $s_i$ and $s_j$, $\gamma_i>\gamma_j$ if and only if $\eta_i>\eta_j$. 
This means that changing the special disk $s_0$ does not change the relative orders among these associated elements. 
Without loss of generality, we choose the special disk $s_0$ to be a disk whose associated group element has minimal order among all the disks in $\mathcal{S}$.  Note that there may be more than one disk with minimal order, so the choice of $s_0$ is not unique. 
Since $\gamma_0=1$, we now have $\gamma_i\ge 1$ in $\pi_1(M)$ for all the disks $s_i\in\mathcal{S}$. 

Let $g_1$ and $g_2$ be a pair of loops in the handlebody $\mathcal{U}$  based at the point $X$ that transversely intersect $U_1$ and $U_2$ respectively in a single point. 
We also assume that $g_1\cap U_2=g_2\cap U_1=\emptyset$.  
To simplify notation, we also use $g_1$ and $g_2$ to denote the corresponding elements $[g_1]$ and $[g_2]$ in $\pi_1(M, X)$. Since $M$ is not a lens space, $g_1$ and $g_2$ are non-trivial elements in $\pi_1(M)$.  
By switching to $g_i^{-1}$ if necessary, we may choose the direction of $g_i$ so that $g_i$ has positive order i.e.~$g_i>1$ in $\pi_1(M)$, $i=1,2$. 
We fix a normal direction for each disk $U_i$ that agrees with the direction of the loop $g_i$.
Each curve $\partial U_i$ has an induced normal direction in the Heegaard surface $S$.

\begin{lemma}\label{Ldirection}
Suppose a pair of disks $s_p$ and $s_q$ ($0\le p, q\le n$) in $\mathcal{S}$ are adjacent and share an edge $e$ and suppose $e\subset\partial U_i$ for some $i=1,2$. 
If the induced normal direction of $e$ points from $s_p$ to $s_q$, then $\gamma_q>\gamma_p$ in $\pi_1(M)$.
\end{lemma}
\begin{proof}
 As illustrated in Figure~\ref{Finequal}(a), the loop $\gamma_p\cdot g_i$ is homotopic to the loop $\gamma_q$, i.e. $\gamma_q=\gamma_p g_i$ in $\pi_1(M)$. Since $g_i>1$ and this is a left order, we have $\gamma_q=\gamma_p g_i>\gamma_p$.  
\end{proof}

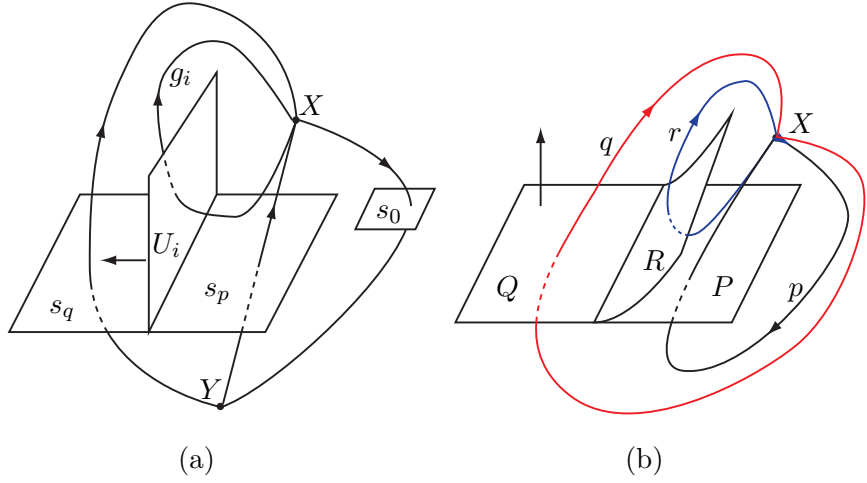
\begin{figure}[h]
	\begin{overpic}[width=4.5in]{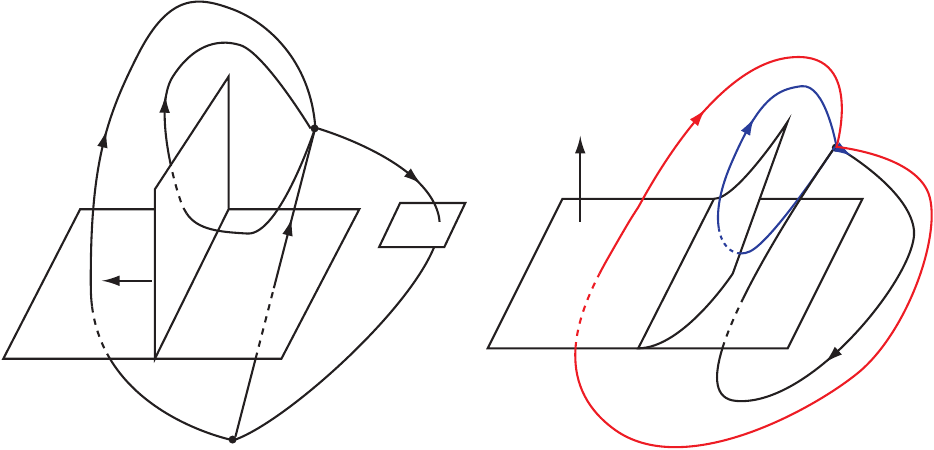}
		\put(34,35){$X$}
		\put(91,33){$X$}
		\put(22.5,2){$Y$}
		\put(5,12){$s_q$}
		\put(23,14){$s_p$}
		\put(43,23){$s_0$}
		\put(17,19){$U_i$}
		\put(19,39){$g_i$}
		\put(82,14){$P$}
		\put(57,14){$Q$}
		\put(74,17){$R$}
		\put(91,14){$p$}
		\put(69,31){$q$}
		\put(77,32){$r$}
		\put(20,-6){(a)}
		\put(72,-6){(b)}
	\end{overpic}
	\vskip 0.5cm
	\caption{Associated loops of adjacent disks}\label{Finequal}
\end{figure}

Now we consider the handlebody $\mathcal{V}$ and the pair of disks $V_1$ and $V_2$.  Recall that the point $Y$ lies in $\mathcal{V}\setminus(V_1\cup V_2)$ and $l_0$ is a path connecting $X$ to $Y$ through $s_0$.  
Let $h_1'$ and $h_2'$ be loops in $\mathcal{V}$ based at $Y$ such that $h_i'$ transversely intersects $V_i$ in a single point and $h_1'\cap V_2=h_2'\cap V_1=\emptyset$. Let $h_i$ be the loop $l_0\cdot h_i'\cdot l_0^{-1}$ based at $X$, $i=1,2$. 
Again, since $M$ is not a lens space, both $h_1$ and $h_2$ are essential loops in $M$. 
We may choose the orientations of $h_1'$ and $h_2'$ so that both $h_1$ and $h_2$ are positive elements, i.e.~$h_i>1$ for both $i$. 
Assign a normal direction to each disk $V_i$ that agrees with the orientation of $h_i'$.

Next we deform the 2-complex $S\cup U_1\cup U_2\cup V_1\cup V_2$ into a branched surface. First, assign a normal direction to $S$ that points from $\mathcal{V}$ to $\mathcal{U}$. Then, deform the 2-complex into a transversely oriented branched surface so that its transverse orientation agrees with the normal directions of $U_1$, $U_2$, $V_1$, $V_2$ and $S$. For example, the local picture of the 2-complex in Figure~\ref{Finequal}(a) is deformed into a branched surface  in Figure~\ref{Finequal}(b). 
Denote the resulting branched surface by $B$.  

\begin{lemma}\label{Lnoclosed}
The only closed and connected surface carried by $B$ is the Heegaard surface $S$.
\end{lemma}
\begin{proof}
Suppose $F$ is a closed and connected surface carried by $B$. Then $F$ has a normal direction induced from the transverse orientation of $B$.  This also implies that $F$ must be orientable.

Since the disks $U_i$ and $V_i$ are non-separating in the respective handlebodies $\mathcal{U}$ and $\mathcal{V}$, and since $F$ admits a normal direction compatible with the transverse orientation of $B$, if $F$ passes through any disk sector $U_i$ or $V_j$, then there is a path connecting the plus side of $F$ to its minus side in $M\setminus F$, which implies that $F$ is non-separating in $M$ and $H_1(M;\mathbb{Q})$ is non-trivial.  Since we have assumed that $H_1(M;\mathbb{Q})$ is trivial in Assumption~\ref{Assumption}, $F$ does not pass through any disk $U_i$ or $V_j$ and this means that $F$ is parallel to the Heegaard surface $S$.
\end{proof}

By our construction, the disks $U_1$, $U_2$, $V_1$, $V_2$ and $s_i$ ($i=0,\dots, n)$ are the branch sectors of $B$, and the branched sectors are associated with loops $g_1$, $g_2$, $h_1$, $h_2$ and $\gamma_i$ ($i=0,\dots, n)$ respectively. Except for $\gamma_0$, each loop transversely intersects the corresponding branched sector exactly once. Moreover, $g_1$, $g_2$, $h_1$, $h_2$ are all positive elements in $\pi_1(M)$ and each $\gamma_i$ is either positive or trivial in $\pi_1(M)$.

For each branch sector $P$ of $B$, we call the two sides of $P$ plus and minus sides, where the transverse orientation at $P$ points from the minus side to the plus side.

\begin{lemma}\label{Lineq}
Let $P$, $Q$, and $R$ be branch sectors of $B$ that meet locally at an arc in the branch locus. Suppose the normal directions of $P$, $Q$, and $R$ are as in Figure~\ref{Finequal}(b), i.e.~ $R$ is on the plus side of $P\cup Q$ and the branch direction (at the branch locus $\partial P\cap\partial Q$) points from $P$ to $Q$. Let $p$, $q$ and $r$ be the elements of $\pi_1(M)$ associated with $P$, $Q$ and $R$ respectively. Then $q=pr$ and $q\ge p$. In particular, $q=p$ if and only if $r=1$ in $\pi_1(M)$.
\end{lemma}
\begin{proof}
The proof is the same as that of Lemma~\ref{Ldirection}. Since $R$ is on the plus side of $P\cup Q$, it is clear from Figure~\ref{Finequal}(b) that $q=pr$. Since the associated element for each branch sector is non-negative, we have $r\ge 1$.  Since this is a left order, $q=pr\ge p$. Moreover, $q=p$ if and only if $r=1$ in $\pi_1(M)$.
\end{proof}

\begin{lemma}\label{Lsource}
Suppose $\gamma_i$ is trivial in $\pi_1(M)$. Then the branch direction at each boundary edge of $s_i$ must point out of $s_i$. 
\end{lemma}
\begin{proof}
Note that $\gamma_0=1$ in $\pi_1(M)$ by construction. So the lemma is not vacuous. 

Let $e$ be a boundary edge of $s_i$ and let $s_j$ be the disk in $\mathcal{S}$ that shares the boundary edge $e$ with $s_i$. Since  the element associated with each branch sector is non-negative and by the hypothesis $\gamma_i=1$, we have $\gamma_j\ge \gamma_i=1$ in $\pi_1(M)$. 

The edge $e$ is an arc in either $\partial U_m$ or $\partial V_m$ ($m=1,2$). 
First suppose $e\subset\partial U_1$ (the case that $e\subset \partial U_2$ is similar). 
If the branch direction at $e$ points into $s_i$ and hence out of $s_j$, by Lemma~\ref{Ldirection}, we have $\gamma_i>\gamma_j$, contradicting the conclusion $\gamma_j\ge \gamma_i$ above. 

It remains to consider the possibility that $e\subset\partial V_1$ (the case that $e\subset \partial V_2$ is similar).  If the branch direction at $e$ points into $s_i$ and hence out of $s_j$, by Lemma~\ref{Lineq}, we have $\gamma_i=h_1 \gamma_j$. By the hypothesis $\gamma_i=1$, we have $h_1=\gamma_j^{-1}$.  However, by our construction, $\gamma_j\ge 1$ for any $j$ and hence $h_1=\gamma_j^{-1}\le 1$. This is a contradiction since we have chosen $h_1$ to be positive in our construction.
\end{proof}

\begin{lemma}\label{Lsourcequad}
Suppose $\gamma_i$ is trivial in $\pi_1(M)$. Then $s_i$ is a quadrilateral with its 4 boundary edges belong to different disks $U_1$, $U_2$, $V_1$ and $V_2$.
\end{lemma}
\begin{proof}
By our construction, the boundary of $s_i$ has an even number of edges, alternately from $\{\partial U_1,\partial U_2\}$ and $\{\partial V_1,\partial V_2\}$.  
Since the intersection of the curves $\partial U_i$ and $\partial V_j$ is tight, $s_i$ cannot be a bigon. 
If the boundary of $s_i$ has more than 4 edges, then there must be two distinct boundary edges $e_1$ and $e_2$ that belong to the same disk $U_j$ or $V_j$, say $U_1$.  
By Lemma~\ref{Lsource}, the branch directions at both $e_1$ and $e_2$ point out of $s_i$. This implies that an arc in $s_i$ connecting $e_1$ to $e_2$ is a wave with respect to $U_1$, contradicting our no-wave assumption on the Heegaard diagram in section~\ref{Ssetup}. Thus, the boundary of $s_i$ consists of exactly 4 edges that are subarcs of different disks $U_1$, $U_2$, $V_1$ and $V_2$. 
\end{proof}
\begin{remark*}
In the construction above, we have chosen the base disk $s_0$ so that $\gamma_i\ge 1$ for each sector $s_i$. If we choose an arbitrary base disk $s_0$, then Lemma~\ref{Lsource} and Lemma~\ref{Lsourcequad} say that any sector with the least order must be a quadrilateral with branch direction at its boundary pointing outwards.
\end{remark*}
Suppose $s_0, \dots, s_k$ are all the disks in $\mathcal{S}$ whose associated elements are trivial in $\pi_1(M)$. By Lemma~\ref{Lsource} and Lemma~\ref{Lsourcequad}, $s_0, \dots, s_k$ are quadrilaterals as shown in Figure~\ref{Fsource}(a).

Note that if a branch sector $P$ has branch directions at all of its boundary edges pointing out of $P$, then as shown in Figure~\ref{Fdelet}, if we remove $\Int(P)$ from $B$, $B\setminus\Int(P)$ is a new branched surface.  By Lemma~\ref{Lsource}, we can remove the interior of the disks  $s_0, \dots, s_k$ from $B$ and obtain a new branched surface, see Figure~\ref{Fsource}(b).  Denote the new branched surface by $B_0$.

\begin{figure}[h]
	\begin{overpic}[width=4in]{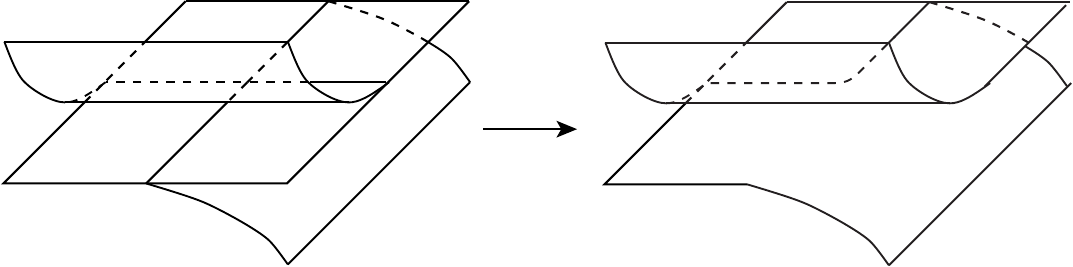}
	\put(22,9){$P$}
	\end{overpic}
	\caption{Remove $\Int(P)$ from $B$}\label{Fdelet}
\end{figure}

\begin{figure}[h]
	\vskip 0.3cm
	\begin{overpic}[width=4.5in]{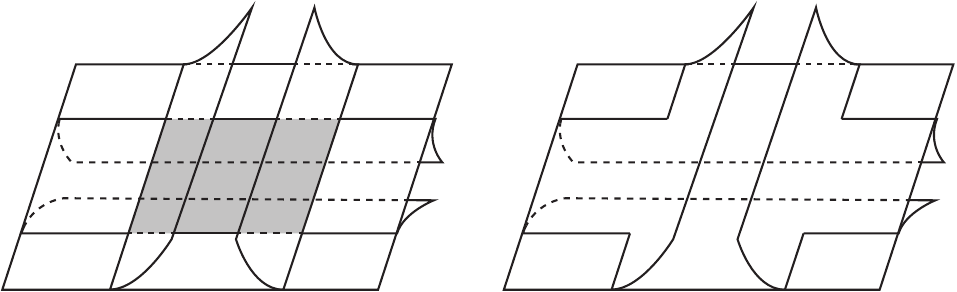}
	\put(20,-5){(a)}
	\put(72,-5){(b)}
	\put(21,28.5){$U_1$}
	\put(33.5,28){$U_2$}
	\put(46,7){$V_1$}
	\put(47,12.5){$V_2$}
	\end{overpic}
	\vskip 0.3cm
	\caption{Remove the disk $s_0$ from $B$}\label{Fsource}
\end{figure}

Different branched sectors of $B$ may merge into the same branch sector of $B_0$. Suppose two branch sectors $P$ and $Q$ of $B$ merge into the same branch sector of $B_0$. Let $p$ and $q$ be the group elements associated with $P$ and $Q$ respectively. 
Since $\gamma_i=1$ for each $i=0,\dots, k$ and by Lemma~\ref{Lineq}, we must have $p=q$ in $\pi_1(M)$.
Thus, we associate the same group element $p=q$ with the newly formed branch sector $P\cup Q$ in $B_0$. We can pick the loop associated with either $P$ or $Q$ for the new branched sector $P\cup Q$.  
So each branch sector of $B_0$ is associated with a loop. 
Since $\gamma_0,\dots,\gamma_k$ are trivial, different choices of the associated loop for any branch sector of $B_0$ represent the same element in $\pi_1(M)$. 
Moreover, this implies that Lemma~\ref{Lineq} also holds for $B_0$

Since the special disk $s_0$ is removed, each associated loop intersects $B_0$ at exactly one point.  Moreover, since we have removed all the disks $s_0, \dots, s_k$ whose associated elements are trivial, $B_0$ has the following properties:
\begin{enumerate}
    \item each branched sector of $B_0$ is associated with a loop that transversely intersects $B_0$ in a single point at this branch sector,
    \item each loop represents a positive element in $\pi_1(M)$
\end{enumerate}
Furthermore,  $B_0$ has an induced transverse orientation and the transverse orientation agrees with the directions of these associated loops.

\section{Construct laminations from a left order}\label{SLF}

In this section, we construct a lamination fully carried by $B_0$. We also discuss a special situation when the lamination can be trivially extended to a co-orientable taut foliation.

\begin{lemma}\label{Llam}
$B_0$ fully carries a lamination.
\end{lemma}
\begin{proof}
The idea of the proof is to perform splittings on $B_0$. The left order on elements in $\pi_1(M)$ provides a splitting instruction. 
The lamination carried by $B_0$ is the inverse limit of an infinite splitting process.

As mentioned in section~\ref{Ssetup}, 
any splitting of a branched surface is a sequence of local splittings along disk neighborhoods of short arcs, as depicted in the 3 possible pictures of Figure~\ref{Fsplit1}, plus the operation of eliminating a disk of contact as in Figure~\ref{Fdisk}.  
Consider the local splitting as illustrated in Figure~\ref{Fsplit2}, where the dashed arrows in Figure~\ref{Fsplit2}(a) denote the transverse orientation and $a$, $b$, $c$, $d$, $e$ denote the elements in $\pi_1(M)$ associated with the corresponding branch sectors. By Lemma~\ref{Lineq}, we have  $ab=e=cd$ in Figure~\ref{Fsplit2}(a). This implies that $a^{-1}c=bd^{-1}$. 
If we perform a splitting as in Figure~\ref{Fsplit2}, then a new branch sector is created by the splitting. 
As shown in Figure~\ref{Fsplit2}(b),  we associate a loop $\delta$  with this new branch sector, such that (1) the loop $\delta$ transversely intersects the branched surface in a single point at this new branch sector, and (2) the orientation of $\delta$ agrees with the transverse orientation of the branched surface. 
It follows from our construction that different choices of such a loop $\delta$ represent the same element in $\pi_1(M)$. 
Moreover, Lemma~\ref{Lineq} also holds for the branched surface after this splitting. 
This means that $c=a\delta$ and $b=\delta d$, see Figure~\ref{Fsplit2}(b). In other words, $\delta=a^{-1}c=bd^{-1}$.  Notice that $\delta>1$ if and only if $c>a$ or equivalently $d^{-1}>b^{-1}$.  

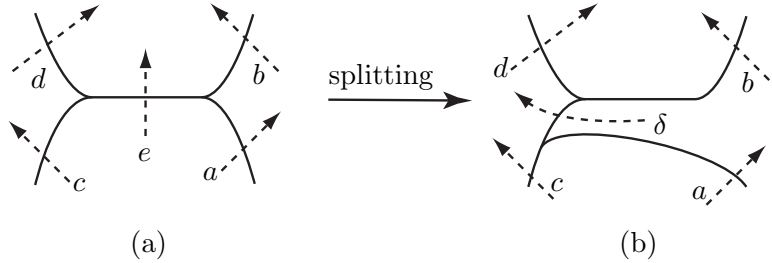
\begin{figure}[h]
	\vskip 0.3cm
	\begin{overpic}[width=4in]{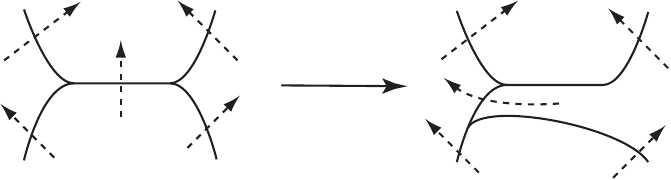}
		\put(16,-6){(a)}
		\put(80,-6){(b)}
		\put(17,6){$e$}
		\put(25.5,3.5){$a$}
		\put(32,16){$b$}
		\put(8.5,2){$c$}
		\put(3,15.5){$d$}
		\put(89.5,1){$a$}
		\put(96,15){$b$}
		\put(71,2){$c$}
		\put(63.5,17){$d$}
		\put(84.5,9.5){$\delta$}
		\put(42,16){splitting}
	\end{overpic}
	\vskip 0.3cm
	\caption{A local splitting}\label{Fsplit2}
\end{figure}

Our splitting instruction is: if $c>a$ or equivalently $d^{-1}>b^{-1}$, we perform the local splitting as in Figure~\ref{Fsplit2} which is the same as the first splitting in Figure~\ref{Fsplit1}. If $c=a$ or equivalently $d^{-1}=b^{-1}$, we perform the splitting as in the second picture in Figure~\ref{Fsplit1}.  If $c<a$ or equivalently $d^{-1}<b^{-1}$, we perform the splitting as in the third picture in Figure~\ref{Fsplit1}. 
If $c>a$ or $c<a$, this splitting creates a new branch sector, but
the splitting instruction guarantees that, for any newly created branch sector, its associated element $\delta$ is positive, and Lemma~\ref{Lineq} always holds. 
If $c=a$, then the splitting drills a tunnel, and the complement of the new branched surface is no longer simply connected. Nonetheless, it follows from $c=a$ that any loop that runs through the tunnel must be homotopically trivial in $M$. 
Thus, for each branch sector, any choice of the associated loop always represents the same element in $\pi_1(M)$.

Note that Figure~\ref{Fsplit1} is a one-dimensional schematic picture, the full picture may include another branch sector as in Figure~\ref{Fsplit3}, which happens if the splitting is along an arc at the branch locus. 
But there is no real difference in the splitting instruction. 

\begin{figure}[h]
	\vskip 0.3cm
	\begin{overpic}[width=3.5in]{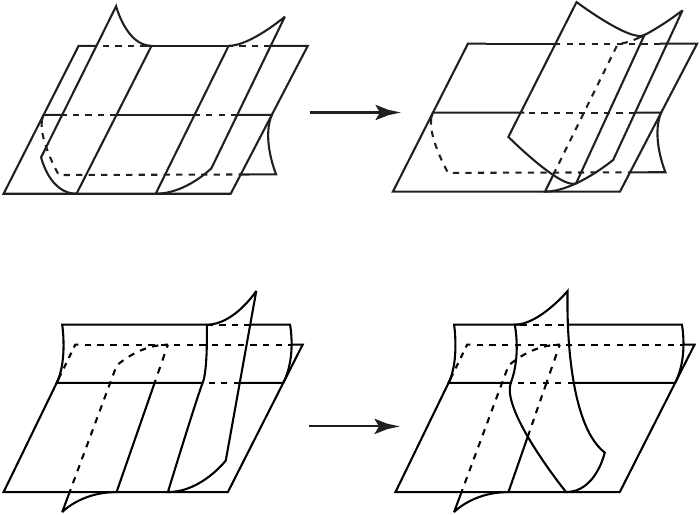}	
	\put(43,59.5){splitting}
	\put(43,14.5){splitting}
	\end{overpic}
	\vskip 0.3cm
	\caption{Splitting along branch locus}\label{Fsplit3}
\end{figure}

We perform such local splittings along (disk neighborhoods of) arcs in all directions. If we see any simple disk of contact in a small 3-ball that consists of two small disks pinched together at a common subdisk as in Figure~\ref{Fdisk}(a), we can perform a splitting along this simple disk of contact as in Figure~\ref{Fdisk}.  

The only local obstruction to continuing such splitting is a twisted disk of contact. 
For example, if we use incompatible splittings at the top and bottom arcs of Figure~\ref{Ftwist1}(a), then we create Figure~\ref{Ftwist1}(b) which is a twisted disk of contact. Twisted disks of contact were introduced in \cite{GO} as a local obstruction for a branched surface to carry a lamination since they force non-trivial holonomy along a trivial curve. 

\begin{figure}[h]
	\vskip 0.3cm
	\begin{overpic}[width=3in]{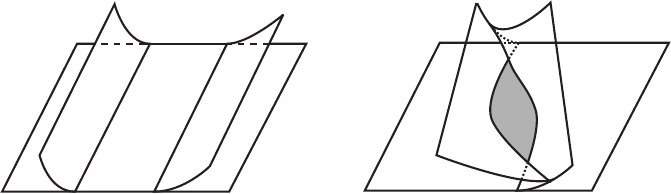}
		\put(16,-6){(a)}
		\put(75,-6){(b)}
	\end{overpic}
	\vskip 0.3cm
	\caption{Create a twisted disk of contact}\label{Ftwist1}
\end{figure}
 
Note that the twisted disk of contact in Figure~\ref{Ftwist1}(b) is equivalent to the twisted disk of contact in Figure~\ref{Ftwist2}(a). In general, any twisted disk of contact can be obtained by locally adding some branch sectors to a simple disk of contact such that the branch directions of these branch sectors are all the same along the boundary of the disk of contact, see Figure~\ref{Ftwist2}(a).  
The simplest twisted disk of contact is Figure~\ref{Ftwist2}(b).

We claim that a twisted disk of contact never appears. 
To see this, first consider Figure~\ref{Ftwist2}(b). Let $c$ be the associated element of the branch sector attached to the disk of contact. The loop $c$ is homotopic in $M$ to the boundary of the disk of contact and hence $c=1$ in $\pi_1(M)$. However, in our construction, the associated elements of all the branch sectors are positive. This is a contradiction.   
Similarly, for Figure~\ref{Ftwist2}(a), denote the associated elements of the two branch sectors attached to the disk of contact by $a$ and $b$, as shown in Figure~\ref{Ftwist2}(a). So the loop around the twisted disk of contact is $a\cdot b$ and hence $ab=1$ in $\pi_1(M)$. Again, this is impossible since both $a>1$ and $b>1$ in our construction. The argument for a twisted disk of contact with an arbitrary number of branch sectors attached to a disk of contact is the same.
Therefore, during our splitting process, a twisted disk of contact never appears, and we can continue the splitting process.

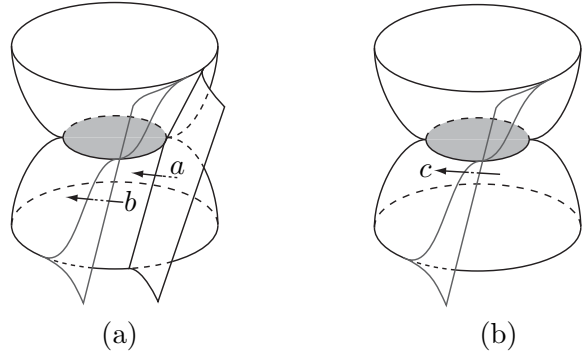
\begin{figure}[h]
	\vskip 0.3cm
	\begin{overpic}[width=3in]{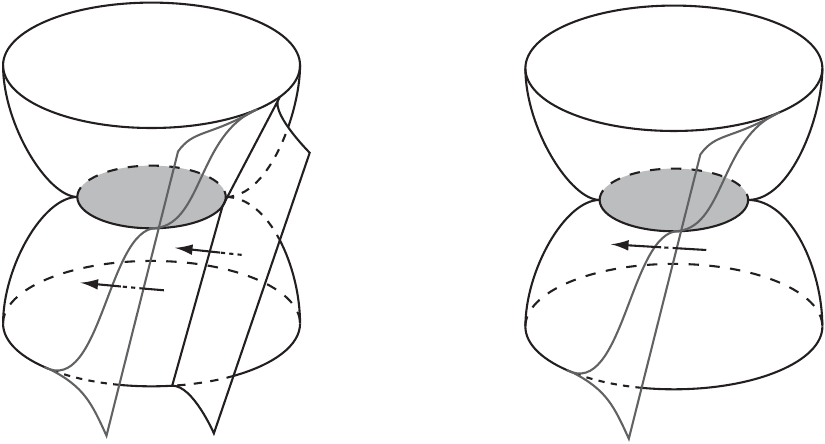}
		\put(16,-6){(a)}
		\put(82,-6){(b)}
		\put(28,23.5){$a$}
		\put(20,17){$b$}
		\put(71.5,23){$c$}
	\end{overpic}
	\vskip 0.3cm
	\caption{Twisted disks of contact}\label{Ftwist2}
\end{figure}

The splitting stops in a finite number of steps only if the splitting is along a compact splitting surface $F\subset N(B_0)$ with $\partial F\subset\partial_vN(B_0)$ and each component of $\partial_vN(B_0)$ contains exactly one boundary curve of $F$. In other words, if we split (or cut) $N(B_0)$ along $F$, there is no more vertical boundary, and the resulting space is an $I$-bundle over a closed surface. 
This means that $B_0$ carries a closed surface.  Any surface carried by $B_0$ is also carried by the original branched surface $B$. 
However, by Lemma~\ref{Lnoclosed}, the Heegaard surface $S$ is the only closed surface carried by $B$. 
Since we have removed at least one quadrilateral disk $s_0$, $B_0$ does not carry the Heegaard surface $S$. This is a contradiction.

Therefore, this splitting process can be continued indefinitely and in all possible directions. Note that one can view that the splitting is along an infinite (not necessarily connected) splitting surface $F$, where $\partial F\subset\partial_vN(B_0)$ and each component of $\partial_vN(B_0)$ contains exactly one boundary curve of $F$. The complement of $F$ in $N(B_0)$ is an interval bundle over a non-compact surface.

As mentioned in section~\ref{Ssetup}, if a branched surface $B_2$ is obtained by splitting $B_1$, we may view $N(B_2)\subset N(B_1)$.   
So the inverse limit of the infinite splitting process is a lamination fully carried by $B_0$ (see \cite{MO}, similar constructions are also used in \cite{L1, L8}).

\end{proof}

The branched surface $B_0$ in Lemma~\ref{Llam} is obtained by removing all the branch sectors $s_0,\dots, s_k$ in $S$ whose associated group elements are trivial. The simplest case is that there is only one such disk $s_0$.

\begin{lemma}\label{Lfoliation}
If there is a unique disk whose associated element is trivial, in other words, if $k=0$ and $B_0=B\setminus\Int(s_0)$, then $M$ admits a co-orientable taut foliation.
\end{lemma}
\begin{proof}
If we remove $\Int(s_0)$ from $B$, the two 3-balls $\mathcal{U}\setminus(U_1\cup U_2)$ and $\mathcal{V}\setminus(V_1\cup V_2)$ merge into a single 3-ball.  Hence $M\setminus\Int(N(B_0))$ is a 3-ball. 
By Lemma~\ref{Lsourcequad}, the 4 boundary edges of $s_0$ belong to different disks $U_1$, $U_2$, $V_1$ and $V_2$. Thus, if we remove $\Int(s_0)$ from $B$, the 4 curves $\partial U_1$, $\partial U_2$, $\partial V_1$ and $\partial V_2$ merge into a single (immersed) curve in $S$, see Figure~\ref{Fsource}.  This means that the branch locus of $B_0$ is a single curve and $\partial_vN(B_0)$ is a single annulus.

The annulus $\partial_vN(B_0)$ divides the boundary sphere of the 3-ball $M\setminus\Int(N(B_0))$ into a pair of disks. Thus, the 3-ball $M\setminus\Int(N(B_0))$ admits a product structure $D^2\times I$ with $(\partial D^2)\times I=\partial_vN(B_0)$. 
By Lemma~\ref{Llam}, $B_0$ fully carries a lamination $\lambda$. 
The conclusion that $M\setminus\Int(N(B_0))$ is a product implies that the complement of $\lambda$ is an $I$-bundle. Hence $\lambda$ can be trivially extended to a foliation $\mathcal{F}$.  
Moreover, since $B_0$ is transversely oriented, the foliation is co-orientable (or transversely orientable).  Furthermore, the loops associated with the branched sectors can be isotoped to be transverse to $\mathcal{F}$ and these loops meet every leaf of $\mathcal{F}$.  Hence $\mathcal{F}$ is a taut foliation.
\end{proof}

By Lemma~\ref{Lfoliation}, it remains to consider the case that there are at least two disks whose associated elements are trivial, i.e.~the case $k\ge 1$. 

Suppose $k\ge 1$. Let $B_0'=B\setminus\Int(s_0)$. 

\begin{corollary}\label{Conemin}
If $B_0'$ fully carries a lamination, then the lamination can be extended to a co-orientable taut foliation of $M$.
\end{corollary}
\begin{proof}
This directly follows from the proof of Lemma~\ref{Lfoliation}.
\end{proof}

By Lemma~\ref{Llam}, $B_0$ fully carries a lamination $\lambda$. 
We may view $B_0'$ as the branched surface obtained by adding back the branch sectors $s_1,\dots, s_k$ to $B_0$. Our next goal is to show that we can extend $\lambda$ to a lamination fully carried by $B_0'$ and by Corollary~\ref{Conemin}, $M$ admits a co-orientable taut foliation.

\section{Extend laminations to foliations}

In this section, we finish the proof of Theorem~\ref{Tmain}, using a proposition that will be proved in the next section.

Consider the set of quadrilateral disks $\{s_0,\dots, s_k\}$ whose associated group elements are trivial in $\pi_1(M)$.   
Let $\mathfrak{s}$ be a subset of $\{s_0,\dots, s_k\}$ and suppose $s_0\in\mathfrak{s}$. Let $\mathfrak{B}$ be the branched surface obtained by removing the interior of the disks in $\mathfrak{s}$ from $B$.  

The proof of Theorem~\ref{Tmain} is basically an induction on  $\mathfrak{s}$. 
By Lemma~\ref{Llam}, if $\mathfrak{s}=\{s_0,\dots, s_k\}$ (i.e.~$\mathfrak{B}=B_0$), then $\mathfrak{B}$ fully carries a lamination.  Next, we suppose $\mathfrak{B}$ fully carries a lamination $\lambda$.

The idea of the proof is to show that either $\lambda$ can be extended to a taut foliation, or there is a proper subset $\mathfrak{s}'\subsetneq\mathfrak{s}$ such that the new branched surface $\mathfrak{B}'$, constructed by removing the interior of the disks in $\mathfrak{s}'$ from $B$, fully carries a lamination. Note that $\mathfrak{B}'$ can be obtained by adding the disks in $\mathfrak{s}\setminus\mathfrak{s}'$ to $\mathfrak{B}$. 
This induction process stops when $\mathfrak{s}'$ contains only one disk $s_0$, i.e.~$\mathfrak{B}'=B_0'$, and Theorem~\ref{Tmain} follows from  Corollary~\ref{Conemin}.

By Lemma~\ref{Lfoliation}, we may suppose $|\mathfrak{s}|>1$, that is, $\mathfrak{s}$ contains more than one disk.

Before we proceed, we would like to describe some standard operations and facts about laminations. The first operation is to replace a leaf in a lamination with an $I$-bundle over this leaf. One can then delete the interior of this $I$-bundle.  
Another useful fact is that if the branch locus of a branched surface has no double point, i.e.~the branch locus is a collection of disjoint simple curves, then the branched surface always fully carries a lamination. The lamination can be constructed as follows: for each branch sector $D$, take a product of $D$ with a Cantor set, then one can glue these product laminations together along the branch locus since the union of two Cantor sets is also a Cantor set. See \cite{L3} for a general discussion. 

Since $\mathfrak{B}$ fully carries the lamination $\lambda$, after some isotopy, we may assume $\partial_hN(\mathfrak{B})\subset\lambda$. Let $M_\lambda$ be the closure of $M\setminus\lambda$ under the path metric. 
As $\partial_hN(\mathfrak{B})\subset\lambda$, the annuli of $\partial_vN(\mathfrak{B})$ are properly embedded in $M_\lambda$ dividing $M_\lambda$ into two parts:  $M\setminus\Int(N(\mathfrak{B}))$ and the infinitesimal region $M_\lambda'$, where   
$M_\lambda'$ lies in $N(\mathfrak{B})$ and hence is a collection of $I$-bundles.  
Since $\mathfrak{B}$ is transversely oriented, each component of $M_\lambda'$ must be a product $F\times I$, where $F$ is a possibly non-compact surface. 
The annuli of $\partial_vN(\mathfrak{B})$ are the vertical boundary of $M_\lambda'$ in the form of $\partial F\times I$. 

If two annuli $A_1$ and $A_2$ of $\partial_vN(\mathfrak{B})$ lie in the same component $N=F\times I$ of $M_\lambda'$, then as illustrated in Figure~\ref{Fadd}, we can first take a surface in $N$ in the form of $F\times\{1/2\}$, and then push its boundary circles onto $F\times\{0\}$ and $F\times\{1\}$, forming a branched surface that separates $A_1$ from $A_2$ in $N$. The leaves of $\lambda$ that contain $F\times\{0\}$ and $F\times\{1\}$, together with this piece of surface, form a possibly non-compact branched surface, see Figure~\ref{Fadd}.  
As described above, this non-compact branched surface fully carries a lamination since its branch locus has no double point. Moreover, using the operations described above, we can add the leaves of this lamination to $\lambda$ and obtain a new lamination. After finitely many such operations, we may assume that our lamination $\lambda$ has the property that no two annuli of $\partial_vN(\mathfrak{B})$ lie in the same component of the infinitesimal region $M_\lambda'$.

\begin{figure}[h]
	\vskip 0.3cm
	\begin{overpic}[width=4in]{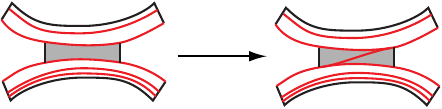}
	\put(11,11){$N$}	
	\end{overpic}
	\caption{Separate two annuli of $\partial_vN(\mathfrak{B})$}\label{Fadd}
\end{figure}

\begin{lemma}\label{Lnodisk}
If the branched surface $\mathfrak{B}$ has no disk of contact with respect to $\lambda$, then $M$ admits a co-orientable taut foliation.
\end{lemma}
\begin{proof}
Let $E=M\setminus\Int(N(\mathfrak{B}))$.  The boundary of $E$ has 3 parts: 
\begin{enumerate}
    \item the components of $\partial_hN(\mathfrak{B})$ with induced normal direction pointing out of $E$, which we denote by $R_+$;
    \item the components of $\partial_hN(\mathfrak{B})$ with induced normal direction pointing into $E$, which we denote by $R_-$;
    \item the annuli in $\partial_vN(\mathfrak{B})$.
\end{enumerate}
One may view $E$ as a sutured manifold with $\partial_vN(\mathfrak{B})$ being its suture, see \cite{G1, G2, G3}. 

The branched surface $B_0'=B\setminus\Int(s_0)$ can be obtained by adding the disks in $\mathfrak{s}\setminus s_0$ back to $\mathfrak{B}$. Our goal is to show that the lamination $\lambda$ can be extended to a lamination fully carried by $B_0'$.

We may view the disks in $\mathfrak{s}\setminus s_0$ as properly embedded disks in $E=M\setminus\Int(N(\mathfrak{B}))$, and this is similar to the sutured manifold decomposition \cite{G1, G2, G3}. 
Let $C$ be the union of $R_+$, $R_-$ and all the disks in $\mathfrak{s}\setminus s_0$. 
Similar to the construction of $B$, we can deform $C$ into a transversely oriented branched surface $B_E$ whose transverse orientation agrees with the normal directions of $R_\pm$ and the disks in $\mathfrak{s}\setminus s_0$. 
$B_E$ is a branched surface with boundary and $\partial B_E\subset\partial_vN(\mathfrak{B})$.  

Since the disks $s_1,\dots, s_k$ are disjoint, the branch locus of $B_E$ has no double point. As explained before Lemma~\ref{Lnodisk}, this implies that $B_E$ fully carries a lamination $\lambda_E$. 

Now glue $R_+$ and $R_-$ back to $\partial_hN(\mathfrak{B})$. The restriction of the lamination $\lambda_E$ on $\partial_vN(\mathfrak{B})$ is a one-dimensional lamination of the annuli $\partial_vN(\mathfrak{B})$. Moreover, the one-dimensional lamination is transverse to the $I$-fibers.

Let $A$ be a component of $\partial_vN(\mathfrak{B})$ and let $M_A$ be the component of the infinitesimal region $M_\lambda'$ that contains $A$. 
So $M_A=F_A\times I$. Since we have assumed at the beginning that no two annuli of $\partial_vN(\mathfrak{B})$ belong to the same component of $M_\lambda'$. The surface $F_A$ has a single boundary circle and $A=(\partial F_A)\times I$.
Moreover, since $\mathfrak{B}$ has no disk of contact with respect to $\lambda$, $F_A$ is not a disk. 

Let $\mu_A$ be the restriction of $\lambda_E$ at $A$. So $\mu_A$ is a  one-dimensional lamination of $A$.

If $M_A$ is compact, since $F_A$ is not a disk, $F_A$ has genus at least one.  For a surface of genus at least one, any 1-dimensional foliation or lamination of $A=(\partial F_A)\times I$ that is transverse to the $I$-fibers can be extended to a 2-dimensional foliation or lamination of $F_A\times I$ (a key reason for this is that any $f\in Homeo_+(I)$ is a commutator, e.g.~see \cite[Lemma 3.1 and Lemma 3.2]{L3} and \cite{G1, G4}). Thus we can extend $\mu_A$ to a 2-dimensional lamination in $M_A$.

If $M_A$ is not compact, then $F_A$ is a non-compact surface with at least one open end. In this case, one can also extend $\mu_A$ to a 2-dimensional lamination in $M_A$.  To see this,  first consider the subcase that $F_A$ is a half-open annulus $S^1\times [0,\infty)$. In this subcase, we can extend $\mu_A$ to the lamination $\mu_A\times [0,\infty)$ in $M_A$.  
Any other surface type for $F_A$ can be obtained by taking punctures and adding genus to the half-open annulus $S^1\times [0,\infty)$.  
By the discussion in the case that $M_A$ is compact, and by \cite[Lemma 3.1 and Lemma 3.2]{L3}, we can take punctures on the product lamination $\mu_A\times [0,\infty)$ and extend the lamination through the higher genus portion.  So in all cases, $\mu_A$ can be extended to a 2-dimensional lamination in $M_A$.

By extending the lamination $\lambda_E$ from $\partial_vN(\mathfrak{B})$ into the infinitesimal region $M_\lambda'$, we obtain a lamination $\lambda'$. By our construction of $B_E$, $\lambda'$ is fully carried by $B_0'=B\setminus\Int(s_0)$.  Now the lemma follows from Corollary~\ref{Conemin}. 
\end{proof}

To prove Theorem~\ref{Tmain}, it remains to consider the case that $\mathfrak{B}$ has at least one disk of contact with respect to $\lambda$.

\begin{lemma}\label{Lalldisks}
At least one component of $\partial_vN(\mathfrak{B})$ does not bound a disk of contact.
\end{lemma}
\begin{proof}
Suppose on the contrary that every component of $\partial_vN(\mathfrak{B})$ bounds a disk of contact. 
After some cutting and pasting if necessary, we may assume that the disks of contact bounded by different components of $\partial_vN(\mathfrak{B})$ are disjoint. 

By cutting $N(\mathfrak{B})$ along these disks of contact, we eliminate $\partial_vN(\mathfrak{B})$ and obtain an $I$-bundle over a closed surface. This means that $\mathfrak{B}$ carries a closed surface.  
Any surface carried by $\mathfrak{B}$ is also carried by $B$. 
However, by Lemma~\ref{Lnoclosed}, the only closed and connected surface carried by $B$ is the Heegaard surface $S$, which moreover, is supported on the union of the sectors $s_i$. 
Since $\mathfrak{B}$ is obtained by removing at least one disk $s_0$ from $B$, $\mathfrak{B}$ does not carry the Heegaard surface $S$. This is a contradiction.
\end{proof}

 Let $A_1,\dots, A_q$ be all the annuli in $\partial_vN(\mathfrak{B})$ that bound disks of contact with respect to $\lambda$. 
 By Lemma~\ref{Lalldisks}, we have $1\le q<|\partial_vN(\mathfrak{B})|$. 
 Let $D_1,\dots D_q$ be a set of disks of contact with $\partial D_i\subset A_i$ ($i=1,\dots, q$) and $D_i\cap\lambda=\emptyset$. Moreover, the disks $D_1,\dots, D_q$ are disjoint in $N(\mathfrak{B})$.

Without loss of generality, suppose $\mathfrak{s}=\{s_0,\dots, s_m\}$ ($m\le k$).

In the proof of the next lemma, we need Proposition~\ref{Pbigon} in the next section.  The proof of Proposition~\ref{Pbigon} is surprisingly subtle and uses techniques very different from other parts of the paper. So we leave it to the last section.

\begin{lemma}\label{Latmostk}
Suppose $\mathfrak{s}=\{s_0,\dots, s_m\}$ ($m\le k$). 
The number of disks of contact is at most $m$, in other words, $q\le m$.   
\end{lemma}
\begin{proof}
Recall that the branched surface $\mathfrak{B}$ is obtained by removing the disks in $\mathfrak{s}$ from $B$. By Lemma~\ref{Lsourcequad}, each disk $s_i$ ($i=0,\dots, m$) in $\mathfrak{s}$ is a quadrilateral with 4 corners.  Each corner of $s_i$ corresponds to a point in the branch locus of $\mathfrak{B}$ that connects an arc from $\partial U_i$ to an arc from $\partial V_j$ ($i, j= 1 \text{ or }2$). We call this point in the branch locus a corner point, see Figure~\ref{Fsource}.

Denote the branch locus of $\mathfrak{B}$ by $L$.  
Each component of $L$ (i.e. the projection of an annular component of  $\partial_vN(\mathfrak{B})$) consists of an even number of arcs alternately from $\{u_1, u_2\}$ and $\{v_1,  v_2\}$ ($u_i=\partial U_i$ and $v_i=\partial V_i$). 
In other words, each component of $L$ contains an even number of corner points. Moreover, it follows from Lemma~\ref{Lsourcequad} that each component of $L$ contains at least two corner points.

If a component $l$ of $L$ contains exactly 2 corner points, then $l$ consists of an arc from $u_i$ and an arc from $v_j$.  
Note that the two corner points must belong to two different disks in $\mathfrak{s}=\{s_0,\dots, s_m\}$, since the 4 boundary edges of each disk in $\mathfrak{s}$ belong to 4 distinct curves $u_1$, $u_2$, $v_1$, $v_2$, see Lemma~\ref{Lsourcequad} and Figure~\ref{Fsource}. 
Suppose $l$ contains exactly 2 corner points and suppose these 2 corner points are vertices of $s_0$ and $s_1$. 
Then $l$ can be viewed as a component of the branch locus of the branched surface $B_1=B\setminus (\Int(s_0)\cup \Int(s_1))$.  
If $l$ bounds a disk of contact in $\mathfrak{B}$, then $l$ also bounds a disk of contact in $B_1$ since $B_1$ can be obtained by adding some quadrilateral disks to $\mathfrak{B}$.  
By Proposition~\ref{Pbigon} in the next section, $B_1$ does not have such a bigon disk of contact. This implies that if $l$ bounds a disk of contact in $N(\mathfrak{B})$, then $l$ contains at least 4 corner points.

Since $\mathfrak{s}$ consists of $m+1$ quadrilateral disks, there are $4(m+1)$ corner points. 
By Lemma~\ref{Lalldisks}, at least one component of $\partial_vN(\mathfrak{B})$ does not bound a disk of contact. 
Thus the number of components of $\partial_vN(\mathfrak{B})$ that bound disks of contact is at most $m$. 
\end{proof}

Consider the sub-manifold $E=M\setminus\Int(N(\mathfrak{B}))$ of $M$. By the construction of $\mathfrak{B}$, $E$ can be viewed as the manifold obtained by connecting the two 3-balls $\mathcal{U}\setminus(U_1\cup U_2)$ and $\mathcal{V}\setminus(V_1\cup V_2)$ through the $m+1$ disks $s_0,\dots, s_m$ in $\mathfrak{s}$. 
Hence, $E$ is a handlebody of genus $m$.  

Now consider the disks of contact $D_1,\dots, D_q$ in $N(\mathfrak{B})$. 
Let $N(D_i)=D_i\times I$ be a small product neighborhood of $D_i$ in $N(\mathfrak{B})$.  We may view $N(D_i)$ as a 2-handle added to $E$. 
Let $\widehat{E}$ be the union of $E$ and these 2-handles $N(D_1),\dots, N(D_q)$.  

Let $\widehat{\mathfrak{B}}$ be the branched surface obtained by splitting $\mathfrak{B}$ along these disks of contact $D_1,\dots, D_q$. So $N(\widehat{\mathfrak{B}})$ can be obtained by cutting $N(\mathfrak{B})$ along $D_1,\dots, D_q$. 
Moreover,  $M\setminus\Int(N(\widehat{\mathfrak{B}}))=\widehat{E}$. 
Since we have assumed before Lemma~\ref{Lnodisk} that no two annuli in $\partial_vN(\mathfrak{B})$ lie in the same component of the infinitesimal region $M_\lambda'$ and since $D_i\cap\lambda=\emptyset$, $\widehat{\mathfrak{B}}$ also fully carries $\lambda$.

\begin{lemma}\label{Lball}
If $\partial\widehat{E}$ consists of 2-spheres, then $q=m$ and $\widehat{E}$ must be a 3-ball in the form of $D^2\times I$.  Moreover, $\lambda$ can be extended to a co-orientable taut foliation.
\end{lemma}
 \begin{proof}
 $\widehat{E}$ is obtained by adding 2-handles to $E$. 
 By Lemma~\ref{Latmostk}, the number of 2-handles is at most $m$ (i.e.~$q\le m$). Since $E$ is a handlebody of genus $m$, if  $\partial\widehat{E}$ consists of 2-spheres, since $\chi(\partial\widehat{E})=2-2m+2q\leq 2$, we have $q=m$ and $\partial\widehat{E}$ is a single 2-sphere.  
Next, we suppose $q=m$.

\begin{claim*}
$\partial_vN(\mathfrak{B})$ has $m+1$ components and $\partial_vN(\widehat{\mathfrak{B}})$ has one component.
\end{claim*}
\begin{proof}[Proof of the claim] 
In the proof of Lemma~\ref{Latmostk}, we see that the quadrilaterals of $\mathfrak{s}$ have total $4(m+1)$ corners. 
If a component of $\partial_vN(\mathfrak{B})$ bounds a disk of contact, then it contains at least 4 corner points.  Since $q=m$,  there are at most 4 corner points not in the boundary of a disk of contact.  
As each component of the branch locus contains at least two corner points, 
this means that there are either one or two components of $\partial_vN(\mathfrak{B})$ that do not bound disks of contact, in other words, $\partial_vN(\widehat{\mathfrak{B}})$ has one or two components.
Next, we rule out the possibility that $\partial_vN(\widehat{\mathfrak{B}})$ has two components.

Since $\partial\widehat{E}$ is a 2-sphere, if $\partial_vN(\widehat{\mathfrak{B}})$ has two components, then the two annuli in $\partial_vN(\widehat{\mathfrak{B}})$ divide the 2-sphere $\partial\widehat{E}$ into 3 components: 2 disks and one annulus, in other words, $\partial_hN(\widehat{\mathfrak{B}})$ consists of two disks and one annulus. 
As $\widehat{E}$ is obtained by adding 2-handles to $E=M\setminus\Int(N(\mathfrak{B}))$, 
this implies that $\partial_hN(\mathfrak{B})$ has 3 components, all of which are planar surfaces. 

We can view $M\setminus\Int(N(\mathfrak{B}))$ as a sutured manifold, where  $\partial_vN(\mathfrak{B})$ is the suture. 
As in the proof of Lemma~\ref{Lnodisk}, a component of  $\partial_hN(\mathfrak{B})$ is in the plus (resp.~minus) boundary of $M\setminus\Int(N(\mathfrak{B}))$ if the induced normal direction points out of (resp.~into)  $M\setminus\Int(N(\mathfrak{B}))$.
Without loss of generality, we may suppose the plus boundary of $\partial_hN(\mathfrak{B})$ has one component and the minus boundary of $\partial_hN(\mathfrak{B})$ has two components.
Similarly, $\widehat{E}$ can also be viewed as a sutured manifold. So the plus boundary of $\widehat{E}$ is an annulus and the minus boundary of $\widehat{E}$ consists of two disks.

Recall that $B$ is the union of $S$ and the disks $U_1$, $U_2$, $V_1$, $V_2$. 
Now consider $N(B)$. We divide $\partial_hN(B)$ into plus and minus boundary similar to $\partial_hN(\mathfrak{B})$. 
$M\setminus B$ consists of two 3-balls, one in $\mathcal{U}$ and one in $\mathcal{V}$. The closure of the 2 components of $M\setminus B$ in $\mathcal{U}$ and $\mathcal{V}$ are of the form $\Delta_\mathcal{U}\times I$ and $\Delta_\mathcal{V}\times I$ respectively, where $\Delta_\mathcal{U}$ and $\Delta_\mathcal{V}$ are disks, $\partial\Delta_\mathcal{U}\times\partial I$ and $\partial\Delta_\mathcal{V}\times\partial I$ are the 4 cusp circles of the branched surface $B$. 
By the transverse orientation of $B$, the plus boundary of $M\setminus\Int(N(B))$ corresponds to the two disks $\Delta_\mathcal{U}\times\partial I$ and the annulus $\partial\Delta_\mathcal{V}\times I$. 

The branched surface $\mathfrak{B}$ is obtained by deleting the quadrilateral disks in $\mathfrak{s}$ from $B$.  The two sides of each disk in $\mathfrak{s}$ are rectangles in the two annuli $\partial\Delta_\mathcal{U}\times I$ and  $\partial\Delta_\mathcal{V}\times I$ respectively. 
Thus, the plus boundary of $M\setminus\Int(N(\mathfrak{B}))$ is obtained by connecting the two disks $\Delta_\mathcal{U}\times\partial I$ with rectangles in the annulus $\partial\Delta_\mathcal{V}\times I$, see Figure~\ref{Fsource}(b) for a local picture. 

We have concluded earlier that each component of $\partial_hN(\mathfrak{B})$ is a planar surface, and we have assumed that the plus boundary of $\partial_hN(\mathfrak{B})$ has one component.  
Since there are $m+1$ disks in $\mathfrak{s}$,  the two disks in $\Delta_\mathcal{U}\times\partial I$ are connected with $m+1$ rectangles. In other words, the plus boundary of $\partial_hN(\mathfrak{B})$ can be obtained by connecting two disks using $m+1$ rectangular (2-dimensional) 1-handles. 
Since the  plus boundary of $\partial_hN(\mathfrak{B})$ is a connected planar surface, this implies that the number of boundary circles of the plus boundary of $\partial_hN(\mathfrak{B})$ is $m+1$. 
As $\widehat{E}$ is obtained by adding $q$ 2-handles to $M\setminus \Int(N(\mathfrak{B}))$ and since $q=m$, the plus boundary of $\widehat{E}$ is obtained by capping off $m$ boundary circles of the plus boundary of $\partial_hN(\mathfrak{B})$. 
This means that the plus boundary of $\widehat{E}$ must have a single boundary circle, contradicting our assumption earlier that the plus boundary of $\widehat{E}$ is an annulus and hence has two boundary circles. 
This finishes the proof of the claim.
\end{proof}

Since $M$ is irreducible, the 2-sphere $\partial\widehat{E}$ bounds a 3-ball in $M$. There are two possibilities.  The first possibility is that $\widehat{E}$ is a 3-ball. Since $\partial_vN(\widehat{\mathfrak{B}})$ has only one component, this implies that the 3-ball  $\widehat{E}$ is a product in the form of $D^2\times I$ and $(\partial D^2)\times I=\partial_vN(\widehat{\mathfrak{B}})$. 
Since $\widehat{\mathfrak{B}}$ also fully carries the lamination $\lambda$, similar to the proof of Lemma~\ref{Lfoliation}, the complement of $\lambda$ in $M$ consists of $I$-bundles, which means that $\lambda$ trivially extends to a foliation. As in the proof of Lemma~\ref{Lfoliation}, this is a co-orientable taut foliation and the lemma holds.
 
 The remaining possibility is that $M\setminus\widehat{E}$ is a 3-ball.  In this case, $\pi_1(\widehat{E})=\pi_1(M)$.  
 However, $\pi_1(E)$ (and hence $\pi_1(\widehat{E})$) is generated by the loops $\gamma_1,\dots, \gamma_m$ associated with the disks $s_1,\dots, s_m$  (see section~\ref{Sbr}).   
 By the definition of $\mathfrak{s}=\{s_0,\dots, s_m\}$ ($m\le k$),  $\gamma_1,\dots, \gamma_m$ represent the  trivial element of $M$. 
 This implies that $\pi_1(\widehat{E})=\pi_1(M)$ is a trivial group, a contradiction to our hypotheses on $M$. 
 \end{proof}

In light of Lemma~\ref{Lball}, we may suppose that not all components of $\partial\widehat{E}$ are 2-spheres.  Hence $H_1(\widehat{E},\mathbb{Q})$ is non-trivial. This implies that the group $\pi_1(\widehat{E})$ is left-orderable, e.g.~see \cite{BRW}, and $\pi_1(\widehat{E})$ is not the trivial group.

Fix a left order $<_E$ for $\pi_1(\widehat{E})$ and consider the orders of the associated elements $\gamma_0,\dots,\gamma_m$ in $\pi_1(\widehat{E})$. 
Recall that, in section~\ref{Sbr}, the loops $\gamma_0,\dots,\gamma_m$ are constructed using a special disk $s_0$. 
Similar to the argument in section~\ref{Sbr}, a different choice of this special disk does not change the relative orders of these associated elements in $\pi_1(\widehat{E})$. 
Hence, as in section~\ref{Sbr}, we choose the special disk $s_0$ to be a disk whose associated element has minimal order (among $\{\gamma_0,\dots,\gamma_m\}$) in $\pi_1(\widehat{E})$.  
Since $\gamma_0=1$ in $\pi_1(\widehat{E})$ and since the order of $\gamma_0$ is minimal, we have $\gamma_i\ge_E 1$ in $\pi_1(\widehat{E})$ for all $i=0,\dots, m$. 

As in the proof of Lemma~\ref{Lball}, $\pi_1(\widehat{E})$ is generated by $\gamma_1,\dots, \gamma_m$.  
Since $\pi_1(\widehat{E})$ is not trivial, at least one element $\gamma_i$ ($1\le i\le m$) is non-trivial in $\pi_1(\widehat{E})$. 
Without loss of generality, suppose $\gamma_1,\dots,\gamma_p$ are non-trivial in $\pi_1(\widehat{E})$ and $\gamma_i=1$ in $\pi_1(\widehat{E})$ if $i>p$ or $i=0$. 
Thus, $\gamma_1,\dots,\gamma_p$ are positive elements in $\pi_1(\widehat{E})$.

\begin{lemma}\label{Ladd}
Let $\mathfrak{B}'$ be the branched surface obtained by adding back the branch sectors $s_1,\dots,s_p$ to $\mathfrak{B}$. Then $\mathfrak{B}'$ fully carries a lamination.
\end{lemma}
\begin{proof}
The proof is a combination of the proofs of Lemma~\ref{Llam} and Lemma~\ref{Lnodisk}.

Recall that $\mathfrak{B}$ fully carries a lamination $\lambda$. After some isotopy on $\lambda\subset N(\mathfrak{B})$, we may assume $\partial_hN(\mathfrak{B})\subset\lambda$. 

First consider $E=M\setminus\Int(N(\mathfrak{B}))$.  
Similar to the proof of Lemma~\ref{Lnodisk}, the boundary of $E$ has 3 parts: 
\begin{enumerate}
	\item the components of $\partial_hN(\mathfrak{B})$ with induced normal direction pointing out of $E$, which we denote by $R_+$;
	\item the components of $\partial_hN(\mathfrak{B})$ with induced normal direction pointing into $E$, which we denote by $R_-$;
	\item the annuli in $\partial_vN(\mathfrak{B})$.
\end{enumerate}
We may view $s_1,\dots, s_p$ as properly embedded disks in $E=M\setminus\Int(N(\mathfrak{B}))$. 
Let $C$ be the union of $R_+$, $R_-$ and the disks $s_1,\dots, s_p$. 
As in the proof of Lemma~\ref{Lnodisk}, deform $C$ into a transversely oriented branched surface. 

Next we split the branched surface $C$ by sliding the disks $s_1,\dots, s_p$ while leaving $R_+$ and $R_-$ invariant.  
The splitting instruction is basically the same as the proof of Lemma~\ref{Llam} except that branch sectors in $R_+$ and $R_-$ do not have any associated group elements in $\pi_1(\widehat{E})$. 
Consider the splittings in Figure~\ref{FsplitE}, where $a$, $b$, $c$ and $d$ denote the loops and elements in $\pi_1(\widehat{E})$ associated with the corresponding branch sectors of $C$. 
Near $R_+$, we perform the splitting as illustrated in Figure~\ref{FsplitE}(a) if $c>_E a$. 
Near $R_-$, we perform the splitting as in Figure~\ref{FsplitE}(b) if $d^{-1}>_E b^{-1}$ (compare with Figure~\ref{Fsplit2}). 
So there is no real difference in our splitting instruction. 
If $c=a$ at $R_+$ (or $d^{-1}=b^{-1}$ at $R_-$), we perform a splitting as illustrated in the middle picture of Figure~\ref{Fsplit1}. 
Recall that $\gamma_1,\dots,\gamma_p$ are positive elements in $\pi_1(\widehat{E})$. 
As explained in the proof of Lemma~\ref{Llam}, following the splitting instruction, any new branch sector created by the splitting is associated with a positive element in $\pi_1(\widehat{E})$.
Similarly,  if $c=a$ or $d^{-1}=b^{-1}$, the splitting changes the complement of the branched surface, but it does not affect the analysis. 

\begin{figure}[h]
	\vskip 0.3cm
	\begin{overpic}[width=4in]{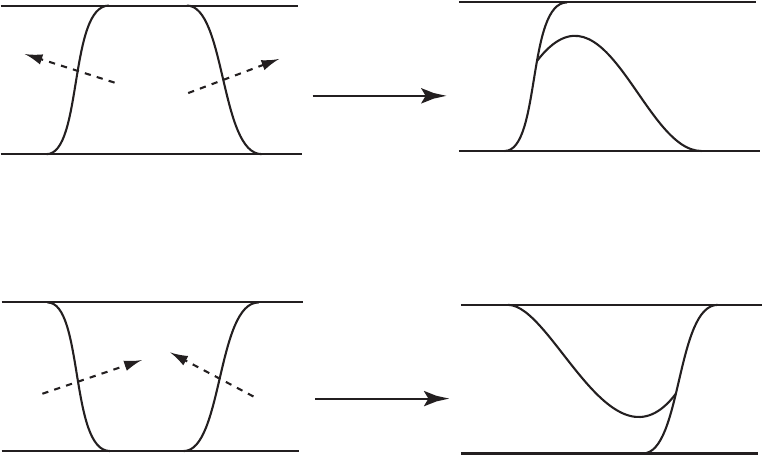}
		\put(47,28){(a)}
		\put(47,-5){(b)}
		\put(14,46){$c$}
		\put(24,44){$a$}
		\put(13.5,7){$d$}
		\put(24,7.5){$b$}
		\put(-5,57){$R_+$}
		\put(-5,18){$R_+$}
		\put(-5,37){$R_-$}
		\put(-5,-2){$R_-$}
		\put(42,50){splitting}
	\end{overpic}
	\vskip 0.3cm
	\caption{Splitting near $R_\pm$}\label{FsplitE}
\end{figure}
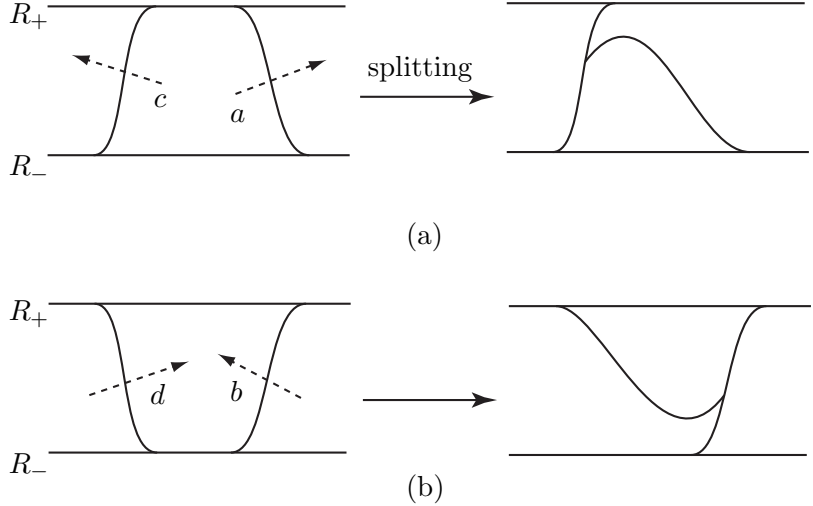

Let $A_i$ be the component of $\partial_vN(\mathfrak{B})$ that bounds the disk of contact $D_i$. 
The core curve of $A_i$ represents the trivial element of $\pi_1(\widehat{E})$. 
We first split the branched surface $C$ at $A_i$. Since each branch sector that is not in $R_\pm$ is associated with a positive element in $\pi_1(\widehat{E})$, similar to the discussion about twisted disks of contact in the proof of Lemma~\ref{Llam}, after some splittings, $C|_{A_i}$ becomes a collection of circles. 
Now we add the 2-handle $N(D_i)$ to $E$ along $A_i$ and cap off each circle in $C|_{A_i}$ using a disk in $N(D_i)$.  After performing this operation for all the annuli $A_1,\dots, A_q$ that bound disks of contact, we obtained a branched surface $\widehat{C}$ in $\widehat{E}$.

As in the proof of Lemma~\ref{Llam}, the splitting instruction from the left order of $\pi_1(\widehat{E})$ allows us to split the branched surface $\widehat{C}$ indefinitely and in all directions. The inverse limit of the splitting process is a lamination fully carried by $\widehat{C}$.  Moreover, this means that the original branched surface $C$ fully carries a lamination $\lambda_C$ whose restriction on each annulus $A_i$ ($i=1,\dots,q)$ of $\partial_vN(\mathfrak{B})$ is a collection of circles.  

The restriction of $\lambda_C$ at each annulus of $\partial_vN(\mathfrak{B})$ is a 1-dimensional lamination transverse to the $I$-fibers. 
Since the restriction of $\lambda_C$ on each $A_i$ ($i=1,\dots,q)$ is a collection of circles, we can cap off these circles by disks in $N(D_i)$.  
By our construction of $\widehat{E}$, no other component of $\partial_vN(\mathfrak{B})$ bounds a disk of contact. 
Thus, as in the proof of Lemma~\ref{Lnodisk}, we can extend the 1-dimensional lamination at $\partial_vN(\mathfrak{B})$ into the infinitesimal region $M_\lambda'$ and obtain a 2-dimensional lamination. 
Similarly, the resulting lamination is fully carried by $\mathfrak{B}'$. 
\end{proof}

Lemma~\ref{Ladd} says that if $\mathfrak{B}$ fully carries a lamination, we can add more disks to $\mathfrak{B}$ and obtain a ``larger" branched surface $\mathfrak{B}'$ which also fully carries a lamination. 
By inductively applying the argument above to the new branched surface $\mathfrak{B}'$, we can continue this process until $\mathfrak{B}'=B_0'=B\setminus\Int(s_0)$ (the choice of the special disk $s_0$ may change during the process).  
Therefore, $B_0'$ fully carries a lamination. By Corollary~\ref{Conemin}, this lamination can be extended to a co-orientable taut foliation.  This proves Theorem~\ref{Tmain}.

\section{Disk of contact}

In this section, we rule out the possibility that our branched surface admits a certain special disk of contact. 

Let $s_0$ and $s_1$ be two disks whose associated elements are trivial. Consider the branched surface $B_1$ obtained by removing the interior of $s_0$ and $s_1$ from $B$, i.e.~$B_1=B\setminus\Int(s_0\cup s_1)$.  By our construction, $M\setminus B_1$ is a solid torus. 
We call the vertices of the two quadrilaterals $s_0$ and $s_1$ the corner points in the branched locus.  
As in the proof of Lemma~\ref{Latmostk}, each component of the branch locus contains an even number of corner points. 
The goal of this section is to prove Proposition~\ref{Pbigon}, which says that if a component of the branched locus contains exactly 2 corner points, then it does not bound a disk of contact in $N(B_1)$. In other words, there is no bigon disk of contact. 

\begin{proposition}\label{Pbigon}
	If a component of $\partial_vN(B_1)$ contains exactly 2 corner points, then it does not bound a disk of contact in $N(B_1)$.
\end{proposition}

Suppose the Heegaard diagram formed by $\{u_1,u_2\}$ and $\{v_1, v_2\}$ is minimal as in Definition~\ref{Dmin}.

If a component of the branch locus contains 2 corner points, then it consists of two arcs, one arc from $u_1\cup u_2$ and the other arc from $v_1\cup v_2$. 
As in the proof of Lemma~\ref{Latmostk}, since the 4 edges of each $s_i$ belong to distinct curves (see Lemma~\ref{Lsourcequad}), if a component of the branch locus of $B_1$ contains 2 corner points, then one corner point is in $\partial s_0$ and the other corner point is in $\partial s_1$.

Suppose $B_1$ admits such a disk of contact $D\subset N(B_1)$.  By the construction of $B_1$, we may view $N(B_1)\subset N(B)$ and view $D$ as a bigon disk in $N(B)$ transverse to the $I$-fibers. 

We may suppose that $N(B)$ is obtained by first taking a product $S\times I$, where $S$ is the Heegaard surface, and then attaching  $N(U_1)$ and $N(U_2)$ to $S\times\{1\}$ and attaching  $N(V_1)$ and $N(V_2)$ to $S\times\{0\}$, where $N(U_i)$ and $N(V_i)$ are product neighborhoods of $U_i$ and $V_i$ respectively. Consider the 4 vertical annuli $u_1\times I$, $u_2\times I$, $v_1\times I$, and $v_2\times I$ in $S\times I$.  Since the disk $D$ is a bigon, without loss of generality, we may suppose $\partial D$ consists of two arcs $\mathfrak{u}$ and $\mathfrak{v}$, where $\mathfrak{u}$ is a subarc of $u_1\times\{1\}$ and $\mathfrak{v}$ is an arc properly embedded in the annulus $v_1\times I$, see Figure~\ref{Fproduct}(a, b). The arrows in  Figure~\ref{Fproduct}(a) denote the transverse orientation of $B$.

\begin{figure}[h]
	\begin{overpic}[width=5in]{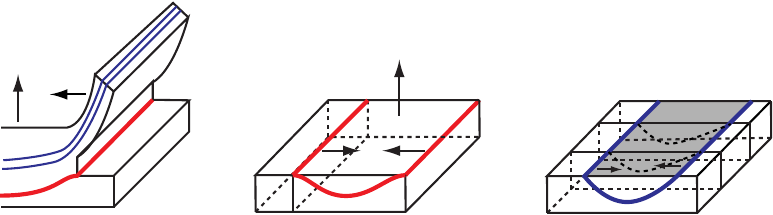}
		\put(11,-5){(a)}
		\put(45,-5){(b)}
		\put(82,-5){(c)}
		\put(17,10){$\mathfrak{u}$}
		\put(7, 1.5){$\mathfrak{v}$}
		\put(43,12){$\mathfrak{u}$}
		\put(49, 1.5){$\mathfrak{v}$}
		\put(56,11){$\mathfrak{u}$}
	\end{overpic}
	\vskip 0.5cm
	\caption{}\label{Fproduct}
\end{figure}

Since the intersection of $u_1\cup u_2$ and $v_1\cup v_2$ is tight, $u_i$ and $v_j$ do not form a bigon in $S$. Hence the disk $D$ cannot be totally in $S\times I$ and $D$ must pass through at least one of the 4 product neighborhoods $N(U_1)$, $N(U_2)$, $N(V_1)$, and $N(V_2)$.  We call each disk of $D\cap(N(U_1)\cup N(U_2))$ a $U$-disk and each disk in $D\cap(N(V_1)\cup N(V_2))$ a $V$-disk.

\begin{lemma}\label{LbothUV}
	$D$ contains at least one $U$-disk and one $V$-disk.
\end{lemma}
\begin{proof}
Suppose $D$ does not have any $U$-disk. Then $D$ can be viewed as a disk in the handlebody $\mathcal{V}$. We may isotope the arc $\mathfrak{v}$ of $\partial D$ into the 3-ball $\mathcal{V}\setminus(V_1\cup V_2)$ and isotope the arc $\mathfrak{u}$ of $\partial D$ to be transverse to the disks $V_1$ and $V_2$. 
Consider the group presentation of $\pi_1(\mathcal{V})$ using the two generators dual to the disks $V_1$ and $V_2$. 
The intersection points of the arc $\mathfrak{u}$ with the disks $V_1$ and $V_2$ give rise to a word $w$ in $\pi_1(\mathcal{V})$ which is the group element represented by $\partial D$. 
Since $D$ lies totally in the handlebody $\mathcal{V}$, $w$ is the trivial element in $\pi_1(\mathcal{V})$. 
Since $\pi_1(\mathcal{V})$ is a free group, there must be two adjacent letters in the word $w$ that cancel each other.  
As $\mathfrak{u}$ is a subarc of $u_1$, 
this implies that a subarc of $\mathfrak{u}$ is a wave connecting the same side of a disk $V_i$ ($i=1$ or $2$). This contradicts our assumption on the Heegaard diagram in section~\ref{Ssetup}, in particular, the assumption that the Heegaard diagram has no wave.

Hence, $D$ must contain a $U$-disk. Symmetrically, the same argument implies that $D$ also contains a $V$-disk.
\end{proof}

Let $\Gamma_1$ and $\Gamma_2$ denote the annuli $v_1\times I$ and $v_2\times I$ respectively.  
In our proof of Proposition~\ref{Pbigon}, we will focus on $D\cap \Gamma_1$ and $D\cap \Gamma_2$.  We can perform a small isotopy  pushing the $V$-disks in $D$ away from $\Gamma_i$, see Figure~\ref{Fvdisk} for a local picture of $\Gamma_i$ and $D$ after this isotopy. So, we may assume that $D\cap \Gamma_1$ and $D\cap \Gamma_2$ consist of $\partial$-parallel arcs with endpoints in $v_1\times\{1\}$ and $v_2\times\{1\}$ respectively, see Figure~\ref{Farcs} for a picture of $\Gamma_i\cap D$.  In particular, each arc in $D\cap \Gamma_i$ ($i=1,2$) cuts off a bigon disk from the annulus $\Gamma_i$.

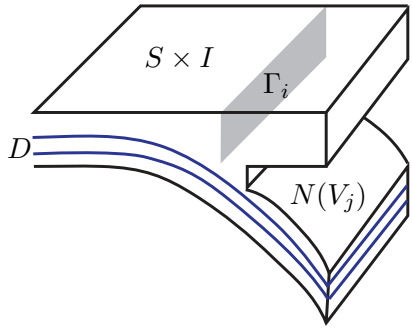
\begin{figure}[h]
	\begin{overpic}[width=2in]{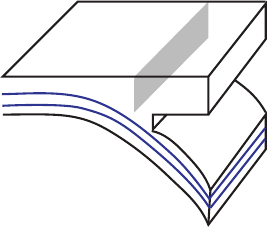}
		\put(61, 61){$\Gamma_i$}
		\put(30, 68){$S\times I$}
        \put(68, 32){$N(V_j)$}
        \put(-6, 44){$D$}
	\end{overpic}
	\caption{$\Gamma_i$ in $N(B)$}\label{Fvdisk}
\end{figure}

\begin{definition}\label{Dshadow}
For each arc $\lambda$ of $D\cap\Gamma_i$, let $\lambda'$ be the arc in $v_i\times\{1\}$ with $\partial\lambda'=\partial\lambda$ and such that $\lambda\cup\lambda'$ bounds a bigon disk in $\Gamma_i$. We call $\lambda'$ the \textit{shadow arc} of $\lambda$ and call the bigon subdisk of $\Gamma_i$ bounded by $\lambda\cup\lambda'$ the \textit{shadow bigon disk} of $\lambda$. 
Furthermore, if there a rectangle $R=I\times I$ in $S\times\{1\}$ such that (1) $\lambda'=\{x\}\times I\subset R=I\times I$ for some $x\in I$, (2) $I\times\partial I\subset D\cap(S\times \{1\})$, and (3) $\partial I\times I\subset (v_1\cup v_2)\times \{1\}\subset \partial\Gamma_1\cup\partial\Gamma_2$, then we call $R$ a \textit{shadow rectangle} for $\lambda$, see the shaded region in Figure~\ref{Fproduct}(c) for a picture.   
Note that we may assume that $R\cap \Gamma_i$ consists of arcs in the form of $\{x\}\times I$ in $R$ and $R\cap D$ consists of arcs in the form of $I\times\{x\}$, $x\in I$.
\end{definition}

Since $D\subset N(B)$ and $D$ is transverse to the $I$-fibers, $D$ 
has a normal direction compatible with 
the transverse orientation of $B$. 
In particular, the normal directions of all the 
$U$-disks are compatible with the transverse orientation of $B$. This normal direction induces a normal direction for each curve of $D\cap (S\times\{1\})$.  In particular, for each arc $\alpha$ in $D\cap \Gamma_i$ ($i=1,2$), the induced directions at both of its endpoints point into the shadow arc of $\alpha$, see the arrows in Figure~\ref{Fproduct}(b, c).

\begin{lemma}\label{Lshadow}
Let $\alpha'$ be a subarc of $(v_1\cup v_2)\times\{1\}$ with $\partial\alpha'\subset D$. If $\Int(\alpha')\cap D=\emptyset$ and the induced normal directions at both endpoints of $\alpha'$ point into $\alpha'$, then $\alpha'$ must be the shadow arc of an arc of $D\cap (\Gamma_1\cup\Gamma_2)$.
\end{lemma}
\begin{proof}
As $\partial\alpha'\subset D$, each point of $\partial\alpha'$ is an endpoint of a component of $D\cap (\Gamma_1\cup\Gamma_2)$.  
The lemma is equivalent to saying that the two endpoints of $\alpha'$ belong to the same component of $D\cap (\Gamma_1\cup\Gamma_2)$.  

Suppose on the contrary that the two endpoints of $\alpha'$ belong to different arcs $\alpha_1$ and $\alpha_2$ of $D\cap (\Gamma_1\cup\Gamma_2)$.  Since the induced normal directions at both endpoints of $\alpha'$ point into $\alpha'$ and since $D$ is transverse to the $I$-fibers, if $\alpha_1\ne\alpha_2$, then  $\Int(\alpha')$ must contain an endpoint of some $\alpha_i$ ($i=1,2$). This contradicts the hypothesis that $\Int(\alpha')\cap D=\emptyset$.
\end{proof}

We call a $U$-disk a $U_i$-disk ($i=1,2$) if the $U$-disk is parallel to the disk $U_i$, i.e.~lying in $U_i\times I$. 
The normal direction gives an induced order on all the $U_i$-disks as follows: we order all the $U_i$-disks as $\Delta_1,\dots,\Delta_m$ so that the induced normal directions point from $\Delta_{j+1}$ to $\Delta_j$ for all $j$, see Figure~\ref{Farcs}, where the shaded region denotes $\Gamma_i$ and the red dot denotes $\mathfrak{u}$. Moreover, we say $\Delta_j$ is in front of $\Delta_{j+1}$ and  $\Delta_{j+1}$ is behind $\Delta_j$.  
We call the disk $\Delta_1$ the \textit{front $U_i$-disk} and call $\partial\Delta_1$ the \textit{front $U_i$-curve}.

\begin{figure}[h]
		\vskip 0.3cm
	\begin{overpic}[width=3in]{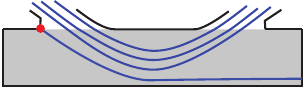}
		\put(15.5, 29){$\Delta_1$}
		\put(4, 29){$\Delta_m$}
		\put(10, 15){$\mathfrak{u}$}
		\put(3, 8){$\Gamma_i$}
		\put(48, 13){$\alpha_1$}
		\put(63, 6.5){$\alpha_m$}
	\end{overpic}
	\caption{}\label{Farcs}
\end{figure}

Recall that $\partial D=\mathfrak{u}\cup \mathfrak{v}$ and  $\mathfrak{u}\subset u_1\times\{1\}$. 
Let $\mathfrak{u}_0=(u_1\times\{1\})\setminus\Int(\mathfrak{u})$
be the complementary arc of $\mathfrak{u}$ in $u_1\times\{1\}$.  
Denote $\mathfrak{u}\cup \mathfrak{u}_0$ by $\delta$. So, at the initial stage, $\delta= u_1\times\{1\}$.

As shown in Figure~\ref{Fproduct}(a) and Figure~\ref{Farcs}, the arc $\mathfrak{u}$ is behind all the $U_1$-disks. 
If $D$ contains a $U_1$-disk,  then $\partial\Delta_1$ (described above) is the front $U_1$-curve; if $D$ has no $U_1$-disk, then we call the curve $\delta=\mathfrak{u}\cup \mathfrak{u}_0$ the front $U_1$-curve.
Thus, there is always a front $U_1$-curve, but if $D$ does not contain a $U_2$-disk, then there is no front $U_2$-curve.

\begin{figure}[h]
	\vskip 0.2cm
	\begin{overpic}[width=5in]{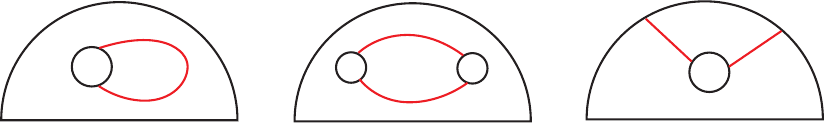}
		\put(12,-5){(a)}
		\put(49,-5){(b)}
		\put(84,-5){(c)}
		\put(15,10.5){$\alpha$}
		\put(9.2,6){$\Delta_s$}
		\put(16,5){$P_s$}
		\put(48,6){$D_\xi$}
		\put(84.5,10.5){$D_\xi$}
	\end{overpic}
	\vskip 0.5cm
	\caption{}\label{FDarcs}
\end{figure}

\begin{lemma}\label{LinnerU}
Let $\alpha$ be an arc in $D\cap\Gamma_i$ and let $\alpha'$ be a subarc of $v_i\times\{1\}$ with $\partial\alpha'=\partial\alpha$. Suppose both endpoints of $\alpha$ lie in the boundary of the same $U$-disk $\Delta_s$ and suppose $\Int(\alpha')\cap\Delta_s=\emptyset$.  Let $P_s$ be the subdisk of $D\setminus\Int(\Delta_s)$ bounded by $\alpha$ and a subarc of $\partial\Delta_s$, see Figure~\ref{FDarcs}(a). Then $P_s$ must contain a $U$-disk.
\end{lemma}
\begin{proof}
Suppose on the contrary that $P_s$ does not contain any $U$-disk.  

The boundary of $P_s$ is the union of $\alpha$ and a subarc  $\kappa$ of $\partial\Delta_s$. 
Recall that the annulus $\Gamma_i$ is a sub-surface of the meridional disk $V_i$. So, by pushing $\alpha$ onto $\alpha'$ across the subdisk of $V_i$ bounded by $\alpha\cup\alpha'$, we can isotope $P_s$ into a disk $P_s'$ with $\partial P_s'=\kappa\cup\alpha'$. 
Since $\Int(\alpha')\cap\Delta_s=\emptyset$, we have $\Int(\alpha')\cap\kappa=\emptyset$, which implies that $\partial P_s'$ is embedded in $S\times\{1\}$. 

If $\partial P_s'$ is a trivial curve in $S\times\{1\}$, then $\kappa$ and the arc $\alpha'$ form a bigon in $S\times\{1\}$, contradicting the assumption that the curves in the Heegaard diagram have tight intersection.  Thus $\partial P_s'$ must be essential in $S\times\{1\}$.  

Next we view $S=S\times\{1\}$ and view $\mathcal{U}$ and $\mathcal{V}$ as the two handlebodies bounded by $S\times\{1\}$.

Since $P_s$ does not contain any $U$-disk, we have $P_s'\subset\mathcal{V}$ after isotopy. 
As $\partial P_s'$ is essential in $S\times\{1\}$, $P_s'$ is an essential disk in $\mathcal{V}$.   
Since $\Int(\alpha')\cap\Delta_s=\emptyset$, after a small isotopy in $S\times\{1\}$, $\partial P_s'$ can be moved to be disjoint from $\partial\Delta_s$.  Since $\Delta_s$ is an essential disk in $\mathcal{U}$ and $P_s'$ is an essential disk in $\mathcal{V}$, this means that the Heegaard splitting is weakly reducible, 
contradicting the conclusion at the end of section~\ref{Ssetup} that the genus-2 Heegaard splitting is strongly irreducible.
\end{proof}

As in the proof of Lemma~\ref{LinnerU}, we view $S=S\times\{1\}$ and view $\mathcal{U}$ and $\mathcal{V}$ as the two handlebodies bounded by $S\times\{1\}$. To simplify notation, we sometimes do not distinguish between $v_i\times\{1\}$ and $v_i$.

\begin{lemma}\label{Lu0}
Suppose $D$ contains a $U_2$-disk. 
Let $\alpha$ be an arc in $D\cap\Gamma_i$ and suppose $\partial\alpha\subset\mathfrak{u}\subset\partial D$.  Let $\alpha'\subset v_i\times\{1\}\subset\partial\Gamma_i$ be an arc with $\partial\alpha'=\partial\alpha$. Then $D\cap\Int(\alpha')\ne\emptyset$. 
\end{lemma}
\begin{proof}
Suppose on the contrary that $D\cap\Int(\alpha')=\emptyset$. 
Let $E$ be a $U_2$-disk in $D$. 
We may regard $\{\delta,\partial E\}$ as the complete set of meridians $\{u_1, u_2\}$ for $\mathcal{U}$, where $\delta=\mathfrak{u}\cup\mathfrak{u}_0$. 
Since $E\subset D$, $\mathfrak{u}\subset\partial D$ and $D\cap\Int(\alpha')=\emptyset$, we have $\Int(\alpha')\cap\partial E=\emptyset$ and $\Int(\alpha')\cap\mathfrak{u}=\emptyset$.  
Note that $\Int(\alpha')$ may still intersect the arc $\mathfrak{u}_0$, so it is possible that $\Int(\alpha')\cap\delta\ne\emptyset$.

Since $\alpha$ is transverse to the $I$-fibers, the induced normal directions at the endpoints of $\alpha$ are opposite along $v_i\times\{1\}$. 
Thus, no matter whether or not $\Int(\alpha')$ meets $\mathfrak{u}_0$, there is always a subarc of $\alpha'$ between two adjacent intersection points of $\alpha'\cap\delta$ and with opposite induced normal directions. Since $\Int(\alpha')\cap\partial E=\emptyset$, such a subarc of $\alpha'$ is a wave with respect to $\delta$, contradicting the property that no subarc of $v_i\times \{1\}$ is a wave with respect to $\{u_1, u_2\}$.
\end{proof}
\begin{remark}\label{Rwave}
	In the proof of Lemma~\ref{Lu0}, we only require that, in the Heegaard diagram, no subarc of $v_1$ and $v_2$ is a wave with respect to $\{u_1, u_2\}$.  Later, we will perform a band sum  and obtain a new set of meridians $\{u_1', u_2'\}$ preserving this property, that is, no subarc of $v_1$ and $v_2$ is a wave with respect to $\{u_1', u_2'\}$, so that we can repeatedly apply Lemma~\ref{Lu0}.
\end{remark}

\begin{lemma}\label{LU12}
Let $\alpha$ be an arc in $D\cap\Gamma_i$  that is outermost in the annulus $\Gamma_i$ ($i=1,2$). 
Then one endpoint of $\alpha$ lies in the front $U_2$-curve and the other endpoint of $\alpha$ lies in the front $U_1$-curve.  In particular, $D$ must contain at least one $U_2$-disk.
\end{lemma}
\begin{proof}
Let $\alpha'\subset v_i\times\{1\}\subset\partial\Gamma_i$ be the shadow arc of $\alpha$.  
Since $\alpha$ is outermost, $\Int(\alpha')\cap D=\emptyset$, in particular, $\Int(\alpha')$ does not intersect any $U$-disk. 
By the ordering of the $U$-disks described after Lemma~\ref{Lshadow}, this implies that each endpoint of $\alpha'$ lies in the front $U_1$- or $U_2$-curve, see Figure~\ref{Farcs}.

We have 4 cases to consider:

\vspace{8pt}
\noindent
Case (i). $D$ contains no $U_2$-disk.
\vspace{8pt}

We will show that this case cannot happen. Suppose $D$ contains no $U_2$-disk. 
Then, by Lemma~\ref{LbothUV}, $D$ must contain at least one $U_1$-disk. This means that both endpoints of $\alpha'$ lie in the boundary of the front $U_1$-disk. 

Let $\Delta_1,\dots,\Delta_m$ be all the $U_1$-disks in $D$ and we order them according to the induced normal directions described earlier, that is, $\Delta_{j+1}$ lies behind $\Delta_j$ for all $j$. So, $\Delta_1$ is the front $U_1$-disk and $\partial\alpha\subset\partial\Delta_1$. 
Since all the $U_1$-disks are parallel, similar to Lemma~\ref{Lshadow} and as shown in Figure~\ref{Farcs}, there must be a sequence of arcs $\alpha_1,\dots,\alpha_m$ (with $\alpha_1=\alpha$) in $D\cap \Gamma_i$ such that  $\partial\alpha_j\subset\partial\Delta_j$ for all $j=1,\dots, m$.

Now consider the arcs $\alpha_1,\dots,\alpha_m$ in $D$. Let $P$ be the planar surface obtained by removing the interior of the $U$-disks $\Int(\Delta_1),\dots, \Int(\Delta_m)$ from $D$. As $\partial\alpha_j\subset\partial\Delta_j$ for each $j$, each $\alpha_j$ cuts off a sub-surface $P_j$ from $P$. Let $\alpha_s$ ($1\le s\le m$) be an innermost arc, that is, the sub-surface $P_s$ cut off by $\alpha_s$ does not contain any other $\alpha_j$.
Since $\partial\alpha_j\subset\partial\Delta_j$,  this means that $P_s$ does not contain any $U$-disk $\Delta_j$.  This contradicts Lemma~\ref{LinnerU} (after setting $\alpha_s$ and the shadow arc of $\alpha_s$ as $\alpha$ and $\alpha'$ in Lemma~\ref{LinnerU}).

As Case (i) cannot happen, we assume next that $D$ contains at least one $U_2$-disk.

\vspace{8pt}
\noindent
Case (ii). $D$ contains no $U_1$-disk and both endpoints of  $\alpha$ lie in the front $U_2$-curve.
\vspace{8pt}

Let $\Delta_1,\dots,\Delta_m$ be all the $U_2$-disks in $D$ and we order them according to the induced normal directions described above, that is, $\Delta_{j+1}$ lies behind $\Delta_j$ for all $j$. In this case, both endpoints of $\alpha'$ lie in the boundary of the front $U_2$-disk $\Delta_1$.  
Now, the proof for this case is identical to the proof of Case (i) after replacing $U_1$ with $U_2$. So, Case(ii) cannot happen. 

\vspace{8pt}
\noindent
Case (iii). $D$ contains no $U_1$-disk and both endpoints of $\alpha$ lie in the front $U_1$-curve. 
\vspace{8pt}

In this case, $\delta=\mathfrak{u}\cup\mathfrak{u}_0$ is the front $U_1$-curve and $\partial\alpha\subset \delta$. 
Thus, $\alpha$ is an arc in $D$ with both endpoints in $\partial D$ and in particular $\partial\alpha\subset \mathfrak{u}$. 
However, since $\alpha$ is outermost, $D\cap\Int(\alpha')=\emptyset$ and this contradicts Lemma~\ref{Lu0}. Hence, Case (iii) cannot occur. 

Thus, if $D$ contains no $U_1$-disk, then one endpoint of $\alpha$ lies in the front $U_1$-curve $\delta$ and the other endpoint of $\alpha$ lies in the front $U_2$-curve, and the lemma holds.

\vspace{8pt}
\noindent
Case (iv). $D$ contains at least one $U_1$-disk.
\vspace{8pt}

In this case, $D$ contains both $U_1$- and $U_2$-disks. 
Let $\Delta_1$ and $E_1$ be the front $U_1$- and $U_2$-disk respectively. 
So the endpoints of $\alpha$ lie in $\partial\Delta_1\cup\partial E_1$. 
Since $\alpha$ is outermost, $\Int(\alpha')\cap\partial\Delta_1=\emptyset$ and $\Int(\alpha')\cap\partial E_1=\emptyset$. 
We may view $\{\partial\Delta_1, \partial E_1\}$ as the complete set of meridians $\{u_1, u_2\}$ of $\mathcal{U}$. 
If both endpoints of $\alpha$ lie in the same front $U_1$- or $U_2$-curve, i.e.~$\partial\Delta_1$ or $\partial E_1$, then $\alpha'$ represents a subarc of $v_1\cup v_2$ in the Heegaard diagram that is a wave with respect to $\{u_1, u_2\}$, contradicting the no-wave hypothesis on the Heegaard diagram.
This implies that one endpoint of $\alpha'$ lies in the front $U_1$-curve $\partial\Delta_1$ and the other endpoint of $\alpha'$ lies in the front $U_2$-curve $\partial E_1$.   Therefore, the lemma holds.
\end{proof}

Suppose that $D$ has $p$ $U_1$-disks, denoted by $\Delta_1,\dots, \Delta_p$, and $q$ $U_2$-disks, denoted by $E_1,\dots, E_q$. 
We order these disks so that $\Delta_{j+1}$ is behind $\Delta_j$ and $E_{j+1}$ is behind $E_j$ for each $j$. 

Let $\delta_i=\partial\Delta_i$. So $\delta_1,\dots,\delta_p$ and $\delta$ are parallel curves and $\delta$ is behind $\delta_p$. 
By Lemma~\ref{LU12}, there is at least one $U_2$-disk. So $q\ge 1$, but it is possible that there is no $U_1$-disk, in which case $p=0$ and $\delta$ is the front $U_1$-curve.
 
We call the two sides of each curve in $D\cap(S\times\{1\})$ the plus and minus sides, where the induced normal direction points from the minus side to the plus side.
 
Let $\widehat{U}'$ be the complete set of meridians of the handlebody $\mathcal{U}$ formed by the front $U_1$- and $U_2$-curves. 
We call an arc $\alpha'\subset (v_1\cup v_2)\times\{1\}$ a \textit{plus-plus} arc with respect to $\widehat{U}'$ if $\alpha'$ connects the plus side of the front $U_1$-curve to the plus side of the front $U_2$-curve $\partial E_1$ and $\Int(\alpha')$ does not intersect any $U_1$- or $U_2$-curve.   
By Lemma~\ref{LHDparallel}, all the plus-plus arcs are parallel in $S\setminus\widehat{U}'$. 

If $p=0$, then a plus-plus arc $\alpha'$ connects $\delta$ to $\partial E_1$.  
As $\delta=\mathfrak{u}\cup\mathfrak{u}_0$, the endpoint $\delta\cap\partial\alpha'$ lies in either $\mathfrak{u}$ or $\mathfrak{u}_0$.

\begin{lemma}\label{Lbothu}
Suppose $p=0$. 
Let $\alpha'$ and $\beta'$ be two plus-plus arcs with respect to $\widehat{U}'$. Let $X=\delta\cap\partial\alpha'$ and $Y=\delta\cap\partial\beta'$ be the endpoints of $\alpha'$ and $\beta'$ in $\delta$ respectively. 
Then $X$ and $Y$ are either both in $\Int(\mathfrak{u})$ or both in $\Int(\mathfrak{u}_0)$.
\end{lemma}
\begin{proof} 
Suppose that $X\notin\Int(\mathfrak{u}_0)$. So $X\in \mathfrak{u}$ and $\partial\alpha'\subset  D$.  
By Lemma~\ref{Lshadow}, $\alpha'$ is the shadow arc of an arc $\alpha$ in $D\cap(\Gamma_1\cap\Gamma_2)$. 
By Lemma~\ref{LU12}, $\alpha$ connects $\partial E_1$ to $\mathfrak{u}$.  
This means that $\alpha\ne\mathfrak{v}$ because $\partial D=\mathfrak{u}\cup\mathfrak{v}$ and $\partial\mathfrak{u}=\partial\mathfrak{v}$. 
Moreover, this also implies that the endpoint $X$ of $\alpha'$ is not an endpoint of $\mathfrak{u}$. Hence $X\in \Int(\mathfrak{u})$.

Let $\gamma_1,\dots,\gamma_a$ be all the plus-plus arcs with respect to $\widehat{U}'$.   So $\alpha'$ and $\beta'$ are among $\gamma_1,\dots,\gamma_a$. 
By Lemma~\ref{LHDparallel},  $\gamma_1,\dots,\gamma_a$ are parallel. 
Let $Z_1,\dots, Z_a$ be the endpoints of $\gamma_1,\dots,\gamma_a$ in $\delta$.  So $X$ and $Y$ are two points in $\{Z_1,\dots, Z_a\}$. 
	
Suppose $\gamma_t=\alpha'$. 
Consider the arc $\gamma_{t\pm 1}$ adjacent to $\gamma_t$. 
Let $R$ be the closure of the rectangle in $S\setminus(\delta\cup\partial E_1)$ between $\gamma_t$ and $\gamma_{t\pm 1}$. 
Let $d$ and $e$ be the other two boundary edges of $R$ with $d\subset\delta$ and $e\subset\partial E_1$. 
Since the endpoint $X$ of $\alpha'$ lies in $\Int(\mathfrak{u})$, the arc $d$ must be totally in $\mathfrak{u}\subset\partial D$.   By Lemma~\ref{Lshadow}, this implies that the arc $\gamma_{t\pm 1}$ must be the shadow arc of a component of $D\cap(\Gamma_1\cup\Gamma_2)$. 
By Lemma~\ref{LU12}, $\gamma_{t\pm 1}$ connects $\partial E_1$ to $\mathfrak{u}$, and this means that the endpoint $Z_{t\pm 1}$ of $\gamma_{t\pm 1}$ lies in $\mathfrak{u}$.  
Moreover, the argument at the beginning of the proof implies that $Z_{t\pm 1}\in\Int(\mathfrak{u})$.
Inductively, we can concludes that $Z_1,\dots, Z_a$ all lie in $\Int(\mathfrak{u})$. In particular $Y\in\Int(\mathfrak{u})$.	
\end{proof}

Let $\rho$ be an outermost arc of $D\cap(\Gamma_1\cup\Gamma_2)$ and let $\rho'$ be its shadow arc.  
If $p=0$, by Lemma~\ref{LU12}, $\rho$ connects $\mathfrak{u}$ to the front $U_2$-disk $E_1$. As $\rho$ is outermost, $\Int(\rho')\cap D=\emptyset$.
Note that, if $p=0$,  $\Int(\rho')$ may still intersect $\mathfrak{u}_0$ even though $\rho$ is outermost.  
We first consider the possibility that $p=0$ and  $\Int(\rho')\cap\mathfrak{u}_0\ne\emptyset$. 
If $\Int(\rho')\cap\mathfrak{u}_0$ has a point $Z$ whose direction (induced from $\delta$) point toward the endpoint of $\rho'$ in $\mathfrak{u}$, then the subarc of $\rho'$ between $Z$ and the endpoint in $\mathfrak{u}$ has induced directions at both endpoints pointing inwards.  
This implies that $\rho'$ has a subarc between two adjacent points of $\rho'\cap\delta$ such that the induced normal directions at its endpoints are opposite. 
Hence, this subarc is a wave with respect to $\delta$, contradicting our assumptions that no subarc of $v_1$ and $v_2$ is a wave with respect to $\{u_1, u_2\}$.  
Thus all the points in $\Int(\rho')\cap\mathfrak{u}_0$ have induced directions pointing toward the endpoint $\partial E_1\cap\partial\rho$.
Hence, a subarc of $\rho'$, denoted by $\alpha'$, is a plus-plus arc connecting $\mathfrak{u}_0$ to $\partial E_1$. 
If this happens, we can perform a band sum of $\mathfrak{u}_0$ with a parallel copy of $\partial E_1$ along $\alpha'$.  
Moreover, by Lemma~\ref{LHDparallel} and Lemma~\ref{Lbothu}, all the plus-plus arcs are parallel to $\alpha'$ and connect $\mathfrak{u}_0$ to $\partial E_1$. Hence the band sum can be extended over all plus-plus arcs. 
This band sum changes the curve $\delta$ into a new meridian $\delta'$ of $\mathcal{U}$. 
By Lemma~\ref{Lbandsum}, no subarc of $(v_1\cup v_2)\times\{1\}$ is a wave with respect to $\{\delta',\partial E_1\}$. 
It is guaranteed by Lemma~\ref{Lbothu} that this band sum only affects $\mathfrak{u}_0$ and the disk $D$ is unchanged. 
After finitely many such band sums on $\mathfrak{u}_0$, $\Int(\rho')\cap\mathfrak{u}_0=\emptyset$ and $\rho'$ becomes a plus-plus arc with respect to $\{\delta',\partial E_1\}$.

\begin{lemma}\label{Lparallel}
Let $\Delta_1,\dots, \Delta_p$ and $E_1,\dots, E_q$ be as above ($q\ge 1$, $p\ge 0$). 
Let $P$ be the planar surface obtained by removing $\Int(\Delta_1),\dots, \Int(\Delta_p)$ and  $\Int(E_1),\dots,\Int(E_q)$ from $D$.  Let $\alpha_1$ be an arc of $D\cap(\Gamma_1\cup\Gamma_2)$ that is outermost in $\Gamma_1\cup\Gamma_2$. Then,
an arc $\eta$ in $D\cap (\Gamma_1\cup\Gamma_2)$ is outermost if and only if $\eta$ is parallel to $\alpha_1$ in the planar surface $P$. 
\end{lemma}
\begin{proof}
First suppose that there is another arc $\eta$ in $D\cap (\Gamma_1\cup \Gamma_2)$ that is parallel to $\alpha_1$ in $P$, i.e., $\eta\cup\alpha_1$ and two subarcs in $\partial P$ form a quadrilateral subdisk $D_\eta$ of $P$.  
We first consider the case that $\eta$ is adjacent to $\alpha_1$, i.e., $\Int(D_\eta)$ is disjoint from $\Gamma_1\cup\Gamma_2$.   
We can perform an isotopy on $D_\eta$ pushing $\alpha_1$ and $\eta$ across their shadow bigon disks onto their shadow arcs in $(v_1\cup v_2)\times\{1\}$. 
This isotopy pushes $D_\eta$ into a disk $D_\eta'$, with $\partial D_\eta'\subset S\times\{1\}$. 
Since $\alpha_1$ is outermost and since $\alpha_1$ and $\eta$ are adjacent in $P$, after a small perturbation, $D_\eta'$ is disjoint from $(v_1\cup v_2)\times\{1\}$.  
Moreover, since $D_\eta$ does not contain any $U$-disk or $V$-disk, $D_\eta'\subset S\times I$.  
This means that $D_\eta'$ must be a $\partial$-parallel disk in $S\times I$ that is disjoint from $\Gamma_1\cup\Gamma_2$ after isotopy.  
Since the intersection of the curves in the Heegaard diagram is tight, this implies that $D$ does not intersect the interior of the subdisk of $S\times\{1\}$ bounded by $\partial D_\eta'$.
Hence, $\eta$ must be outermost in $\Gamma_1\cup\Gamma_2$. 
Thus, if $\eta$ is parallel and adjacent to $\alpha_1$ in $P$, $\eta$ must also be an outermost arc in  $\Gamma_1\cup\Gamma_2$. 
Inductively, we can conclude that every arc of $D\cap \Gamma_i$ that is parallel to $\alpha_1$ in $P$ must be outermost in  $\Gamma_1\cup\Gamma_2$.

Conversely, suppose an arc $\eta$ of $D\cap \Gamma_i$ is outermost in $\Gamma_i$. Then by Lemma~\ref{LU12}, $\eta$ connects the plus side of the front $U_1$-curve to the plus side of the front $U_2$-curve $\partial E_1$. 
Recall that if $p\ge 1$ then the front $U_1$-curve is $\partial\Delta_1$, and if $p=0$ then there is no $U_1$-disk and $\delta$ is the front $U_1$-curve.

Next, we suppose $p\ge 1$, that is, $\eta$ connects $\partial\Delta_1$ to $\partial E_1$. The argument for the case $p=0$ is almost identical after replacing $\partial\Delta_1$ with $\delta$. 

We may view $\{\partial\Delta_1,\partial E_1\}$ as the   complete set of meridians $\{u_1, u_2\}$ in the Heegaard diagram. 
Let $\alpha_1'$ and $\eta'$ be the shadow arcs of $\alpha_1$ and $\eta$ respectively.  
Since both $\alpha_1$ and $\eta$ are outermost and by Lemma~\ref{LU12}, $\alpha_1'$ and $\eta'$ are subarcs of $v_1\cup v_2$ connecting the plus side of $\partial \Delta_1$ to the plus side of $\partial E_1$. Hence, by Lemma~\ref{LHDparallel}, $\alpha_1'$ and $\eta'$ are parallel in $S\setminus(\partial\Delta_1\cup\partial E_1)$.
Let  $R$ be the closure of the rectangle in $S\setminus(\partial\Delta_1\cup\partial E_1)$ between $\alpha_1'$ and $\eta'$. So $\alpha_1'$ and $\eta'$ are a pair of opposite edges of $\partial R$, and the other pair of edges of $\partial R$ are subarcs of $\partial\Delta_1$ and $\partial E_1$.

The rectangle $R$ may contain other subarcs of $(v_1\cup v_2)\times\{1\}$. Let $\xi'$ be the arc of  $(v_1\cup v_2)\times\{1\}$ in $R$ that is adjacent to $\alpha_1'$. It follows from Lemma~\ref{Lshadow} that $\xi'$ must be the shadow arc of an arc $\xi$ of $D\cap(\Gamma_1\cup\Gamma_2)$.
Let $R_0$ be the sub-rectangle of $R$ between $\xi'$ and $\alpha_1'$.  
So $\alpha_1'$ and $\xi'$ are a pair of opposite edges of $\partial R_0$. 
Let $d$ and $e$ be the other two edges of $\partial R_0$ with $d\subset\partial\Delta_1$ and $e\subset\partial E_1$. 
The 4 arcs $d$, $e$, $\alpha_1$ and $\xi$ form a closed curve in $D$, bounding a subdisk $D_\xi$ of $D$, see Figure~\ref{FDarcs}(b) for a picture. 
Note that $\alpha_1$ and $\xi$ are parallel in $P$ if and only if $D_\xi$ does not contain any $U$-disk or $V$-disk. 

By pushing $\alpha_1$ and $\xi$ onto $\alpha_1'$ and $\xi'$ along their shadow bigons respectively, we can isotope $D_\xi$ into a disk $D_\xi'$ with $\partial D_\xi'=\partial R_0\subset S\times\{1\}$.

Suppose $D_\xi$ contains at least one $U$- or $V$-disk. 
Since all the $U_i$- or $V_i$-disks are parallel in $N(B)$ with consistent normal directions, similar to the proof of Lemma~\ref{Lnoclosed}, there is an arc disjoint from the Heegaard surface $S\times\{1\}$ and connecting the plus side of $D_\xi'$ to its minus side. 
This implies that the 2-sphere $R_0\cup D_\xi'$ is homotopically non-trivial, contradicting the hypothesis that $M$ is irreducible.
Thus $D_\xi$ does not contain any $U$-disk or $V$-disk and this means that $\alpha_1$ and $\xi$ are parallel in $P$.
Thus, we can inductively conclude that $\alpha_1$ and $\eta$ are parallel in $P$. 

It remains to consider the special case that $p=0$, i.e., $D$ contains no $U_1$-disk, in which case, $\alpha_1$ and $\eta$ are arcs connecting $\partial E_1$ to $\mathfrak{u}\subset\partial D$.  
The argument is mostly identical to the case $p\ge 1$ after replacing $\partial \Delta_1$ with the meridian $\delta=\mathfrak{u}\cup\mathfrak{u}_0$. We will use the same notations and only point out the slight difference below.

If $\Int(\alpha_1')\cap\mathfrak{u}_0\ne\emptyset$, then we can perform band sums on $\mathfrak{u}_0$ and with parallel copies of $\partial E_1$ as in the argument on $\rho'$ before this lemma. The band sum does not affect $D$.  So, after possibly some band sums on $\mathfrak{u}_0$, we may suppose $\Int(\alpha_1')\cap\mathfrak{u}_0=\emptyset$ and $\alpha_1'$ is a plus-plus arc with respect to $\{\delta,\partial E_1\}$. Moreover, as in the proof of Lemma~\ref{Lbothu}, $\alpha_1'$ connects  $\Int(\mathfrak{u})$ to $\partial E_1$. 

We use the same notation as the case $p\ge 1$. 
Let  $R$ be the closure of the rectangle in $S\setminus(\delta\cup\partial E_1)$ between $\alpha_1'$ and $\eta'$, and let $\xi'$ be the subarc of $(v_1\cup v_2)\times\{1\}$ in $R$ that is parallel and adjacent to $\alpha_1'$. 
The arc $\xi'$ has one endpoint in $\delta$ and the other endpoint in $\partial E_1$. 
Since the endpoint of $\alpha_1'$ in $\delta$ lies in $\Int(\mathfrak{u})$ and since $\xi'$ is adjacent to $\alpha_1'$, the endpoint of $\xi'$ in $\delta$ must also lie in $\mathfrak{u}\subset\partial D$. 
Hence, by Lemma~\ref{Lshadow}, $\xi'$ must be the shadow arc of an arc $\xi$ of $D\cap(\Gamma_1\cup\Gamma_1)$. 
Similar to the case $p\ge 1$, let $R_0$ be the sub-rectangle of $R$ between $\xi'$ and $\alpha_1'$, and let $d\subset\delta$ and $e\subset\partial E_1$ be edges of $\partial R_0$. 
Since the arc $\xi'$ is adjacent to $\alpha_1'$, 
the edge $d$ does not contain an endpoint of $\mathfrak{v}$, which implies that the subarc $d$ of $\delta$ is totally in $\mathfrak{u}$. Hence, as shown in Figure~\ref{FDarcs}(c), the 4 arcs $d$, $e$, $\alpha_1$ and $\xi$ form a closed curve in $D$, bounding a subdisk $D_\xi$ in $D$. Other parts of the argument are identical to the case $p\ge 1$.  
\end{proof}


\noindent
Step 1. Eliminate the $U_1$-disks.  
\vspace{8pt}

Suppose $p\ne 0$. The first step is to perform isotopies on $D$ pushing part of $D\cap(S\times I)$ out of $S\times I$. When restricted to $S\times I$, the isotopy is basically
a $\partial$-compression on $D\cap(S\times I)$. 
The operation also changes the complete set of meridians $\{u_1, u_2\}$ to a new set of meridians.

Let $\alpha_1$ be an arc in $D\cap \Gamma_i$ that is outermost in the annulus $\Gamma_i$. Since $\delta$ is behind $\delta_1,\dots,\delta_p$ ($\delta_i=\partial\Delta_i$), by Lemma~\ref{LU12}, $\alpha_1$ is an arc connecting $\Delta_1$ to $E_1$. 
Let $\alpha_1'\subset v_i\times\{1\}$ be the shadow arc of $\alpha_1$. Since $\delta$ is behind $\delta_1,\dots,\delta_p$, $\alpha_1'$ does not intersect $\delta$. 
So $\alpha_1'$ is a subarc of $v_i$ in $S\setminus(\partial\Delta_1\cup\partial E_1)$ connecting the plus sides of $\Delta_1$ to the plus side of $E_1$.   
By Lemma~\ref{LHDparallel}, all the subarcs of $v_1\cup v_2$ in $S\setminus(\partial\Delta_1\cup\partial E_1)$ that connect the plus side of $\Delta_1$ to the plus side of $E_1$ are parallel, and by Lemma~\ref{Lshadow}, every such subarc is the shadow arc of an outermost arc of $D\cap (\Gamma_1\cup\Gamma_2)$. 

Let $P$ be the planar surface as in Lemma~\ref{Lparallel}. 
Let $R_P$ be a rectangle in $P$ that contains all the arcs of $D\cap(\Gamma_1\cup\Gamma_2)$ that are parallel to the outermost arc $\alpha_1$ in $P$. 
By Lemma~\ref{Lparallel}, $R_P\cap (\Gamma_1\cup\Gamma_2)$ consists of precisely all the outermost arcs of $D\cap (\Gamma_1\cup\Gamma_2)$. 
Moreover, as shown in Figure~\ref{Fproduct}(c), $R_P$ is parallel to a shadow rectangle $R$ in $S\times\{1\}$ that contains all the shadow arcs of the outermost arcs of $D\cap (\Gamma_1\cup\Gamma_2)$. 

Now we perform a small isotopy pushing $R_P$ first onto $R$ and then out of $S\times I$. The effect of this operation on $D\cap (S\times I)$ is basically a $\partial$-compressing.  
After this operation,  $\Delta_1$ and $E_1$ are merged into a single disk whose boundary curve is a band sum of $\partial\Delta_1$ and $\partial E_1$ along $R$ in $S$, see section~\ref{Ssetup} for a description of band sum. 

By Lemma~\ref{Lparallel}, the rectangle $R$ contains every subarc of $v_1\cup v_2$ that connect the plus side of $\Delta_1$ to the plus side of $E_1$. So, by  Lemma~\ref{Lbandsum}, the curve obtained by this band sum has tight intersection with $v_1$ and $v_2$.

We have 3 cases.

\vspace{10pt}

\noindent 
Case (1). $p>q$, that is, the number of $U_1$-disks is more than the number of $U_2$-disks.
\vspace{8pt}

As shown in Figure~\ref{Farcspq}(a), $D\cap \Gamma_i$ has $q$ parallel arcs $\alpha_1,\dots,\alpha_q$ such that each $\alpha_i$ connects $\Delta_i$ to $E_i$. We may successively perform the operation described above on $D$ using $\alpha_1,\dots,\alpha_q$. These operations merge $\Delta_i$ and $E_i$ into a single disk $E_i'$ for each $i=1,\dots, q$. 
After this operation, $D\cap\mathcal{U}$ has two sets of parallel disks $E_1',\dots, E_q'$ and $\Delta_{q+1},\dots, \Delta_p$.  
As $E_i'$ is obtained by a band sum of $E_i$ and $\Delta_i$, $\{\partial E_i', \partial\Delta_j\}$ is a new complete set of meridians for the handlebody $\mathcal{U}$.  
The disks $E_1',\dots, E_q'$ are parallel and admit compatible normal directions. 
After this step, we call $E_1',\dots, E_q'$ the $U_2'$-disks and the remaining disks $\Delta_{q+1},\dots, \Delta_p$ the $U_1'$-disks.

\begin{figure}[h]
	\vskip 0.3cm
	\begin{overpic}[width=4in]{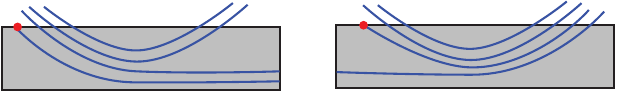}
		\put(20,-5){(a)}
		\put(75,-5){(b)}
		\put(20,7.7){$\alpha_1$}
		\put(75,8){$\alpha_1$}
		\put(33,6.5){$\alpha_q$}
	\end{overpic}
	\vskip 0.5cm
	\caption{}\label{Farcspq}
\end{figure}

\vspace{10pt}
\noindent 
Case (2). $p=q$, i.e.~there is the same number of $U_1$- and $U_2$-disks.
\vspace{8pt} 

In this case, the operation above yield $q$ parallel disks $E_1',\dots, E_q'$, which we still call $U_2'$-disks. There is no more $U_1$-disk after the operations. 
Similar to Case (1), $\{\partial E_i', \delta\}$ is a new complete set of meridians for $\mathcal{U}$.

\vspace{10pt}
\noindent 
Case (3). $q>p$, that is, the number of $U_2$-disks is more than the number of $U_1$-disks.
\vspace{8pt}

This case is slightly more complicated. 
As $q>p$,  $D\cap \Gamma_i$ has $p$ parallel arcs $\alpha_1,\dots,\alpha_p$ such that each $\alpha_j$ connects $\Delta_j$ to $E_j$, see Figure~\ref{Farcspq}(b). 
Let $\alpha_j'\subset v_i\times\{1\}$ be the shadow arc of $\alpha_j$. 
We divide the operation into two steps. The first step is similar to Case (1): By successively pushing rectangles containing $\alpha_1,\dots,\alpha_p$ across the shadow bigons and out of $S\times I$, for each $j$, we merge $\Delta_j$ and $E_j$ into a single disk which we denote by $\Delta_j'$ ($j=1,\dots, p$).  
We call each $\Delta_j'$ ($j=1,\dots, p$) a $U_1'$-disk and call the remaining disk $E_{p+1},\dots, E_q$ the $U_2'$-disks.

The second step is to perform a similar operation on the curve $\delta=\mathfrak{u}\cup \mathfrak{u}_0$.  
Recall the the curve $\delta$ is parallel to and behind $\partial\Delta_p$.  Denote the endpoint of $\alpha_p$ in $\partial\Delta_p$ by $X_p$ and denote the point of $\delta\cap(v_i\times\{1\})$ next and behind $X_p$ by $X_{p+1}$, see Figure~\ref{FXp}.

We have two subcases. The fist subcase is that $X_{p+1}\in \mathfrak{u}$ and the second subcase is that $X_{p+1}\in \Int(\mathfrak{u}_0)$. 
The difference is that if $X_{p+1}\in \mathfrak{u}$, then $X_{p+1}$ is an endpoint of an arc in $D\cap(\Gamma_1\cup\Gamma_2)$, see Figure~\ref{FXp}(a), but if $X_{p+1}\in \Int(\mathfrak{u}_0)$, then no arc of $D\cap(\Gamma_1\cup\Gamma_2)$ is incident to $X_{p+1}$, see Figure~\ref{FXp}(b). Note that, as in the proof of Lemma~\ref{Lbothu}, since $X_{p+1}$ cannot an endpoint of $\mathfrak{v}$, if $X_{p+1}\in \mathfrak{u}$ then $X_{p+1}\in \Int(\mathfrak{u})$.

\begin{figure}[h]
	\vskip 0.2cm
	\begin{overpic}[width=4in]{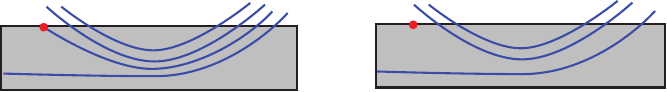}
		\put(20,-5){(a)}
		\put(75,-5){(b)}
	    \put(2.5, 6){$X_{p+1}$}
	    \put(58, 6){$X_{p+1}$}
	\end{overpic}
	\vskip 0.5cm
	\caption{}\label{FXp}
\end{figure}

Before we proceed with the two subcases, we would like to point out that the distinction between the two subcases is consistent for all the outermost bigons. More precisely, suppose there is another arc $\beta_1$ of $D\cap \Gamma_i$ that is outermost in $\Gamma_i$. 
Similar to $\alpha_1,\dots,\alpha_p$ described above, $D\cap\Gamma_i$ has $p$ parallel arcs $\beta_1,\dots,\beta_p$ such that each $\beta_j$ connects $\Delta_j$ to $E_j$. 
Denote the endpoint of $\beta_p$ in $\partial\Delta_p$ by $Y_p$.  As $\delta$ is parallel to and behind $\partial\Delta_p$,  $\delta\cap(v_i\times\{1\})$ has a point, which we denote by $Y_{p+1}$, next to and behind $Y_p$. Similar to $X_{p+1}$, $Y_{p+1}$ lies in either $\mathfrak{u}$ or $\mathfrak{u}_0$.
It follows from the proof of Lemma~\ref{Lbothu} that $X_{p+1}$ and $Y_{p+1}$ are either both in $\Int(\mathfrak{u})$ or both in $\Int(\mathfrak{u}_0)$.

\vspace{8pt}

\noindent 
Subcase (3a). $X_{p+1}\in \Int(\mathfrak{u})$.

In this subcase, since $\mathfrak{u}\subset\partial D$, there is an arc $\alpha_{p+1}$ of $D\cap \Gamma_i$ parallel to $\alpha_p$ and connecting $X_{p+1}$ to $\partial E_{p+1}$, see Figure~\ref{FXp}(a) for a picture of $\alpha_{p+1}$ in $\Gamma_i$ and see Figure~\ref{FEp}(a) for a picture of $\alpha_{p+1}$ in $D$. 
Similar to the operation above, we can perform a $\partial$-compression on $D\setminus\Int(E_{p+1})$ in $S\times I$ along $\alpha_{p+1}$, that is,  pushing a rectangular neighborhood of $\alpha_{p+1}$ in $D\setminus\Int(E_{p+1})$ across the the shadow bigon of $\alpha_{p+1}$ and out of $S\times I$. 
This operation changes the annulus $D\setminus\Int(E_{p+1})$ into a disk $D'$ which can be viewed as a disk obtained by an isotopy on the disk $D\setminus(\Int(E_{p+1})\cup N(\alpha_{p+1}))$, see Figure~\ref{FEp}(b).  In particular, as shown in Figure~\ref{FEp}(b), $\partial E_{p+1}$ and the arc $\mathfrak{u}$ of $\partial D$ merge into a new arc which we denote by $\mathfrak{u}'$. Note that $\partial D'=\mathfrak{u}'\cup\mathfrak{v}$. The curve $\delta$ is merged (or band-summed) with $\partial E_{p+1}$ into a new curve $\delta'$ parallel to and behind $\partial\Delta_p'$, where $\Delta_p'$ is the disk described at the beginning of Case (3).

\begin{figure}[h]
	\vskip 0.2cm
	\begin{overpic}[width=4in]{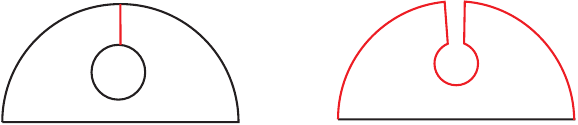}
		\put(18,-5){(a)}
		\put(77,-5){(b)}
		\put(16.5, 8){$E_{p+1}$}
		\put(9.5, 15){$\mathfrak{u}$}
		\put(9, 1){$\mathfrak{v}$}
		\put(67.5, 15.5){$\mathfrak{u}'$}
		\put(21.5,16){$\alpha_{p+1}$}
        \put(87, 7){$D'$}
	\end{overpic}
	\vskip 0.5cm
	\caption{}\label{FEp}
\end{figure}

\vspace{8pt}

\noindent 
Subcase (3b). $X_{p+1}\in \Int(\mathfrak{u}_0)$.

In this subcase, there is no arc of $D\cap \Gamma_i$ connecting $\partial E_{p+1}$ to $X_{p+1}$ since $\mathfrak{u}_0$ is not part of $\partial D$, see Figure~\ref{FXp}(b). 
Take a curve $\varepsilon\subset S\times\{1\}$ parallel to $\partial E_{p+1}$ and lying between $\partial E_p$ and $\partial E_{p+1}$. 
Let $\alpha_{p+1}'$ be the subarc of $v_i\times\{1\}$ connecting $X_{p+1}$ to $\varepsilon$ and with $ \alpha_p'\subset\alpha_{p+1}'$. 
We perform a band sum of $\mathfrak{u}_0$ with $\varepsilon$ along $\alpha_{p+1}'$, which merge $\mathfrak{u}_0$ with $\varepsilon$ into a single arc which we denote by $\mathfrak{u}_0'$. 
The curve $\delta$ is merged with $\varepsilon$ into a new curve $\delta'$ parallel to and behind $\partial\Delta_p'$.
The operation in this subcase is only on $\mathfrak{u}_0$ and it does not affect the disk $D$.

\vspace{8pt}

After the operations above, we have two sets of parallel disks, which we call $U_1'$-disks and $U_2'$-disks, and a new curve $\delta'$ parallel to the $U_1'$-disks (if there is a $U_1'$-disk).  We give a similar order to these $U_1'$- and $U_2'$-disks. In particular, the curve $\delta'$ is behind all the $U_1'$-disks.   

We use $D'$ to denote the disk after these operations and use $\mathfrak{u}'$ and $\mathfrak{u}_0'$ to denote the two arcs in $\delta'$.  In particular, we have 
 $\partial D'=\mathfrak{u}'\cup\mathfrak{v}$ and $\delta'=\mathfrak{u}'\cup\mathfrak{u}_0'$.

In our operation above, we push a rectangle $R_P\subset P$  onto a shadow rectangle $R\subset S\times\{1\}$. The rectangle $R_P$ contains all the outermost arcs of $D\cap(\Gamma_1\cup\Gamma_2)$ and the shadow rectangle $R$ is maximal in the sense that $R$ contains all the arcs in the Heegaard diagram that are parallel to the shadow arc of an outermost arc. 
Thus, by Lemma~\ref{Lbandsum}, the $U_1'$- and $U_2'$-curves satisfy the following properties:

\begin{enumerate}
	\item a $U_1'$-curve and a $U_2'$-curve form a new complete set of meridians for $\mathcal{U}$, and they have tight intersection with $v_1\times\{1\}$ and $v_2\times\{1\}$,
	\item no subarc of $(v_1\cup v_2)\times\{1\} $ is a wave with respect to the new complete set of meridians for $\mathcal{U}$.
\end{enumerate}

We call all the $U_1'$- and $U_2'$-disks the $U'$-disks.  It follows from our construction that $D'$ always contains a $U'$-disk and a $V$-disk after the operation. So Lemma~\ref{LbothUV} remains true for $D'$. 
In the proof of Lemmas~\ref{Lshadow}--\ref{Lparallel}, we did not really use the hypothesis that the Heegaard diagram is minimal in the sense of Definition~\ref{Dmin}. 
The only properties that we need are: (1) the Heegaard splitting is strongly irreducible, and (2) no subarc of $(v_1\cup v_2)\times\{1\} $ is a wave with respect to the complete set of meridians for $\mathcal{U}$, e.g.~see Remark~\ref{Rwave}.  
These properties are not changed by our operations on $D$. Thus, these lemmas also hold for $D'$, $U_1'$, $U_2'$, and $\delta'$. 
Therefore, we can repeat the operations/isotopies until there is no more $U_1'$-disk, in which case the curve $\delta'$ is the front $U_1'$-curve. 

Next, we suppose that there is no more $U_1'$-disk. It follows from Lemma~\ref{LU12} that $D'$ always contains a $U_2'$-disk.

\vspace{8pt}
\noindent
Step 2. Perform operations to reduce the number of $U_2'$-disk to one.

\vspace{8pt}

Consider the disk $D'$ after Step 1 and
let $E_1,\dots, E_q$ be the $U_2'$-disks in $D'$. In this step, we consider the situation $q\ge 2$.

Let $\rho$ be an outermost arc in $D'\cap\Gamma_i$, and let $\rho'\subset v_i\times\{1\}$ be its shadow arc. 
It follows from the proof of Case (ii) and Case (iii) of Lemma~\ref{LU12} that 
$\rho'$ must connect the front $U_2'$-disk to $\mathfrak{u}'\subset\partial D'$. 

Before performing $\partial$-compressions as in Step 1, we need to consider the possibility that $\Int(\rho')\cap\mathfrak{u}_0'\ne\emptyset$. 
This cannot happen in Step 1 because if there is a $U_1$-disk then $\delta$ is not the front $U_1$-curve.
If $\Int(\rho')\cap\mathfrak{u}_0'\ne\emptyset$, then we can apply the operation described before Lemma~\ref{Lparallel}, that is, perform band sum on $\mathfrak{u}_0'$ with parallel copies of $\partial E_1$. It is guaranteed by Lemma~\ref{Lbothu} that this band sum only changes $\mathfrak{u}_0'$ and does not affect the disk $D'$. 
After a finite number of such band sums, $\Int(\rho')\cap\mathfrak{u}_0'=\emptyset$. 

After removing such intersection with $\mathfrak{u}_0'$, we perform a $\partial$-compression along outermost bigons as in Subcase (3a) and obtain a new disk which we still denote by $D'$, see Figure~\ref{FEp}(b).  
The effect of this $\partial$-compression is a band sum that merges $\delta'$ and $\partial E_1$ into a single curve, which we still denote by $\delta'$. The arc $\mathfrak{u}'$ is merged with $\partial E_1$ into a subarc of $\delta'$, which we still denote by $\mathfrak{u}'$.  
As in Subcase (3a) and Figure~\ref{FEp}, this $\partial$-compression reduces the number of $U_2'$-disks by one. 
We continue these operations until $D'$ contains  exactly one $U_2'$-disk.  

In the remaining part of the section, we fix $D'$ and suppose $D'$ contains no $U_1'$-disk and exactly one $U_2'$-disk.

\vspace{8pt}
\noindent
Step 3. Count intersection numbers and obtain a contradiction.

\vspace{8pt}

Consider the new disk $D'$ after Step 2. The $U_1'$- and $U_2'$-curves satisfy the two properties listed before Step 2. 
Let $E_1$ be the remaining $U_2'$-disk in $D'$. By our construction, $D'$ still contains at least one $V$-disk. Let $P$ be the planar surface obtained by removing $\Int(E_1)$ and the interior of all the $V$-disks.  
By Lemma~\ref{Lparallel}, all the outermost arcs of $D'\cap (\Gamma_1\cup\Gamma_2)$ are parallel in $P$ connecting $\partial E_1$ to the arc $\mathfrak{u}'$ in $\partial D'$.

\begin{lemma}\label{LVdisk}
Suppose an arc $\beta$ of $D'\cap\Gamma_i$ has both endpoints in $\mathfrak{u}'$.  Let $D_\beta$ be the subdisk of $D'$ bounded by $\beta$ and the subarc of $\mathfrak{u}'$ between $\partial\beta$. Then $D_\beta$ must contain a $V$-disk.
\end{lemma}
\begin{proof}
The two curves $\partial E_1$ and $\delta'$ form a complete set of meridians of the handlebody $\mathcal{U}$. 
By the properties of $U_1'$- and $U_2'$-curves described before Step 2, the 
intersection of $\partial E_1\cup\delta'$ with $v_1\cup v_2$ is tight, and no subarc of $v_i\times\{1\}$ is a wave with respect to $\{\partial E_1, \delta'\}$.

The argument is similar to the proof of Lemma~\ref{LbothUV}. 
Suppose that $D_\beta$ does not contain any $V$-disk. 
Let $\beta'$ be the shadow arc of $\beta$. 
By pushing $\beta$ onto $\beta'$ along its shadow bigon, we isotope $D_\beta$ into a disk $D_\beta'$ with $\partial D_\beta'\subset S\times\{1\}$. Since $D_\beta$ contains no $V$-disk, $D_\beta'$ lies totally in $\mathcal{U}$ after isotopy. 
Hence $\partial D_\beta'$ is a homotopically trivial loop in $\mathcal{U}$. 
Let $g_1$ and $g_2$ be the two generators of $\pi_1(\mathcal{U})$ dual to the pair of meridians $\{\partial E_1, \delta'\}$. 
Consider the group presentation of $\pi_1(\mathcal{U})$ using the $g_1$ and $g_2$. 
The intersection of $\beta'$ with $\partial E_1$ and $\delta'$ gives a word $w$ in this group presentation and $w$ is the element of $\pi_1(\mathcal{U})$ represented by $\partial D_\beta'$. 
Since $\partial D_\beta'$ is a trivial loop in $\mathcal{U}$ and since $\pi_1(\mathcal{U})$ is a free group, two adjacent letters in the word $w$ can be canceled, i.e.~in the form of $g_j\cdot g_j^{-1}$. This implies that a subarc of $\beta'$ is a wave with respect to $\{\partial E_1, \delta'\}$, contradicting the properties discussed before Step 2.
\end{proof}

In the remaining part of the section, we will measure the complexity of the original Heegaard diagram using $D'$, see Definition~\ref{Dmin}.  

Consider the disks $U_1$ and $U_2$ before all these operations.  
Recall that $N(B)$ is obtained by attaching the products $U_1\times I$, $U_2\times I$, $V_1\times I$ and $V_2\times I$ to $S\times I$.  
We extend the annuli $\Gamma_1$ and $\Gamma_2$ vertically into slightly larger annuli $\widehat{\Gamma}_1$ and $\widehat{\Gamma}_2$ so that $\Gamma_i\subset\Int(\widehat{\Gamma}_i)$ ($i=1,2$) and  $\widehat{\Gamma}_i$ meets part of $U_1\times I$ and $U_2\times I$.  
We call $\widehat{\Gamma}_i\cap(U_1\times I)$ the $U_1$-tails and call $\widehat{\Gamma}_i\cap(U_2\times I)$ the $U_2$-tails of $\Gamma_i$, see Figure~\ref{Ftail}.  
Note that our isotopies/operations push arcs out of   $\widehat{\Gamma}_i$, but the remaining arcs of  $D'\cap \widehat{\Gamma}_i$ are still transverse to the $I$-fibers.  
In particular, the intersection of $D'$ with $\widehat{\Gamma}_1$ and $\widehat{\Gamma}_2$ consists of arcs connecting these $U_1$- and $U_2$-tails, see Figure~\ref{Ftail}.

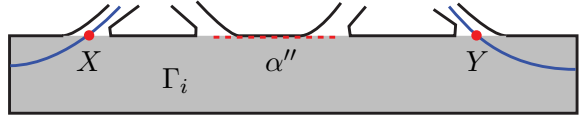
\begin{figure}[h]
	\vskip 0.2cm
	\begin{overpic}[width=3in]{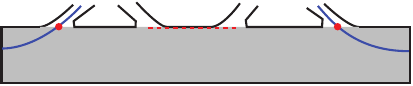}
		\put(12, 8){$X$}
		\put(80, 8){$Y$}
		\put(45,8){$\alpha''$}
	    \put(27,5){$\Gamma_i$}
	\end{overpic}
	\caption{$\Gamma_i$ and $\widehat{\Gamma}_i$}\label{Ftail}
\end{figure}

\begin{lemma}\label{Lminimality} 
Let $k$ be the number of outermost arcs of $D'\cap\Gamma_j$ ($j=1,2$). Then $k\le \min\{|v_j\cap u_1|, |v_j\cap u_2|\}$, where $\{u_1, u_2\}$ and $\{v_1, v_2\}$ are the curves in the original Heegaard diagram.
\end{lemma}
\begin{proof}
Let $\alpha_1,\dots,\alpha_k$ be the outermost arcs of $D'\cap\Gamma_j$. We view these arcs $\alpha_i$'s in $\widehat{\Gamma}_j$ and view each endpoint of $\alpha_i$ as a point in a $U_1$- or $U_2$-tail of $\widehat{\Gamma}_j$ .  Let $X$ and $Y$ be a pair of points in $\{\partial\alpha_1,\dots,\partial\alpha_k\}$ that are adjacent in $v_j\times\{1\}$.  Since the arcs $\alpha_i$'s are outermost and transverse to the $I$-fibers, the normal directions (induced from the transverse orientation of $N(B)$) at $X$ and $Y$ must be opposite along $v_j\times\{1\}$, see Figure~\ref{Ftail}.  This means that $X$ and $Y$ cannot be in the same $U_1$- or $U_2$-tail of $\widehat{\Gamma}_j$. 

As $X$ and $Y$ are adjacent endpoints, there is a subarc $\alpha'$ of $v_j\times\{1\}$ such that $\partial\alpha'= X\cup Y$ and $\alpha'$ does not contain any other endpoint of the $\alpha_i$'s. 
If both $X$ and $Y$ are in $U_2$-tails, we claim that there must be a $U_1$-tail attached to $\Int(\alpha')$. 
To see this, if there is no $U_1$-tail attached to $\alpha'$, since the normal direction at $\partial\alpha'=X\cup Y$ are opposite, there must be a subarc $\alpha''$ (see the dashed red arc in Figure~\ref{Ftail}) of $\alpha'$ such that (1) both endpoints of $\alpha''$ lie in $U_2$-tails, (2) the induced normal directions (or the branch directions) at the endpoints of $\alpha''$ are opposite, and (3) there is no $U_1$- or $U_2$-tail incident to $\Int(\alpha'')$. 
This implies that $\alpha''$ determines a wave with respect to $u_2$ in the original Heegaard diagram formed by $\{u_1, u_2\}$ and $\{v_1, v_2\}$, a contradiction to our no-wave hypothesis on the Heegaard diagram.

The argument above implies that each arc $\alpha'$ between two adjacent endpoints of the $\alpha_i$'s can be associated with at least one $U_1$-tail and one $U_2$-tail, either at the endpoints of $\alpha'$ or in its interior. 
The $\alpha_i$'s have a total of $2k$ endpoints and we have $k$ disjoint arcs like $\alpha'$. 
This implies that the total number of $U_1$-tails and the total number of $U_2$-tails must be both at least $k$.  Hence $|v_j\cap u_1|\ge k$ and $|v_j\cap u_2|\ge k$.
\end{proof}

Let $\beta$ be an arc in $D'\cap \Gamma_i$. 
We say that the arc $\beta$ is \textit{almost outermost} if (1) $\beta$ is not outermost in $\Gamma_i$ and (2) every other arc of $D'\cap \Gamma_i$ in its shadow bigon disk is outermost in $\Gamma_i$. 
By Lemma~\ref{LU12}, the arc $\mathfrak{v}\subset\partial D'$ cannot be an outermost arc.  As $\mathfrak{v}\subset D\cap\Gamma_1$, $D'\cap \Gamma_1$ has at least one almost outermost arc. However, it is possible that every arc of $D'\cap \Gamma_2$ is outermost.

Let $\eta$ be an almost outermost arc in $D'\cap \Gamma_i$ ($i=1$, or $2$), and let $\eta'$ be the shadow arc of $\eta$.   
Since $D'$ contains no $U_1'$-disk and only one $U_2'$-disk $E_1$, by Lemma~\ref{LU12}, all the outermost arcs of $D'\cap(\Gamma_1\cup\Gamma_2)$ are arcs connecting $\partial E_1$ to the arc $\mathfrak{u}'\subset\partial D'$. By Lemma~\ref{Lparallel}, these outermost arcs are all parallel in $P=D'\setminus\Int(E_1)$. 
Let $R_P$ be a small rectangle in $P$ containing all the outermost arcs of $D'\cap(\Gamma_1\cup\Gamma_2)$, see Figure~\ref{Feta}, and let $Q=D'\setminus (E_1\cup R_P)$. 
So $Q$ is a disk. 
As $E_1$ is the only $U'$-disk in $D'$, $P$ and $Q$ do not contain any $U'$-disk. 
By our construction, both $P$ and $Q$ contain a $V$-disk.   
The almost outermost arc $\eta$ is properly embedded in the disk $Q$.  Thus, $\eta$ divides $Q$ into two disks and we denote the disk that does not contain $\mathfrak{v}$ by $D_\eta$. 
$D_\eta$ has 3 possible configurations, depending on the location of the endpoints of $\eta$: 
\begin{enumerate}
	\item If both points of $\partial \eta$ lie in $\mathfrak{u}'$, then $D_\eta$ is as shown in Figure~\ref{Feta}(a). By Lemma~\ref{LVdisk}, $D_\eta$ must contain a $V$-disk.
	\item If both points of $\partial\eta$ lie in $\partial E_1$, then then $D_\eta$ is as shown in Figure~\ref{Feta}(b). Since $\partial E_1$ has tight intersection with $v_i\times\{1\}$, $\eta'$ is not parallel to a subarc of $\partial E_1$ in $S\times I$.  Since $P$ does not contain any $U'$-disk, this means that $D_\eta$ must contain a $V$-disk.
	\item If $\eta$ has one endpoint in $\partial E_1$ and the other endpoint in $\mathfrak{u}'$, then $D_\eta$ is as shown in Figure~\ref{Feta}(c).  Since $\eta$ is not outermost, by Lemma~\ref{Lparallel}, $\eta$ is not parallel to the outermost arcs and cannot be an arc in $R_P$. Since $P$ does not contain any $U'$-disk, $D_\eta$ must also contain a $V$-disk.
\end{enumerate}
Thus, in all possible configurations, $D_\eta$ must contain a $V$-disk.

\begin{figure}[h]
	\vskip 0.2cm
	\begin{overpic}[width=4.8in]{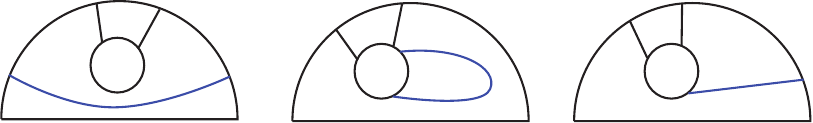}
		\put(12,-5){(a)}
		\put(50,-5){(b)}
		\put(85,-5){(c)}
		\put(13,6){$E_1$}
		\put(45,5){$E_1$}
		\put(81,5){$E_1$}
		\put(13,11.5){$R_P$}
		\put(44,10.5){$R_P$}
		\put(79,10.5){$R_P$}
		\put(5,4.5){$\eta$}
		\put(60.5,3){$\eta$}
		\put(95,3){$\eta$}
		\put(20,7){$D_\eta$}
		\put(52,5){$D_\eta$}
		\put(88,7){$D_\eta$}
	\end{overpic}
	\vskip 0.5cm
	\caption{}\label{Feta}
\end{figure}

We perform two $\partial$-compressions on $D'\cap (S\times I)$. 

First, we perform a $\partial$-compression along outermost arcs of $D'\cap(\Gamma_1\cup\Gamma_2)$, similar to the operations in Subcase (3a) above.   
All the outermost arcs are parallel in $P$ and lie in the rectangle $R_P$. As illustrated in Figure~\ref{Fproduct}(c), $R_P$ is parallel to a shadow rectangle $R$ in $S\times\{1\}$. 
This $\partial$-compression pushes $R_P$ first onto $R$ and then out of $S\times I$. 
Similar to Figure~\ref{FEp}(b), this operation merges $\mathfrak{u}'$ with $\partial E_1$ into an arc $\mathfrak{u}''$. 
As in Subcase (3a) above, the operation yields a new disk (similar to the disk $D'$ in Figure~\ref{FEp}(b)).  To distinguish from $D'$, we denote this new disk by $D''$. 
So $\partial D''=\mathfrak{u}''\cup \mathfrak{v}$. 
After this $\partial$-compression, the almost outermost arc $\eta$ becomes an outermost arc of $D''\cap\Gamma_i$.

Next, we perform another $\partial$-compression on $D''$ along $\eta$, pushing $\eta$ onto its shadow arc $\eta'$. These two $\partial$-compressions push the disk $D_\eta$ described above into a disk $D_\eta'$ with  $\partial D_\eta'\subset S\times\{1\}$. 
As $D''$ contains no $U'$-disk, after isotopy, $D_\eta'$ is a disk properly embedded in the handlebody $\mathcal{V}$. 
Since $D_\eta$ contains at least one $V$-disk, similar to the proof of Lemma~\ref{Lnoclosed}, there is an arc in $\mathcal{V}$ connecting the plus side of $D_\eta'$ to its minus side. This implies that (1)
$\partial D_\eta'$ must be a non-trivial curve in the Heegaard surface $S\times\{1\}$ and (2)  $D_\eta'$ is non-separating in $\mathcal{V}$.  Thus, $\partial D_\eta'$ is a meridian of the handlebody $\mathcal{V}$. 

\begin{lemma}\label{Lalmost}
Let $\eta$ be an almost outermost arc of $D'\cap\Gamma_i$ ($i=1$ or $2$) and let $D_\eta'$ be the disk described above. Let $k_1$ be the number of outermost arcs of $D'\cap\Gamma_i$ in the shadow bigon disk of $\eta$. Then, 
$|\partial D_\eta'\cap\partial E_1|\le k_1<\min \{|v_i\cap u_1|, |v_i\cap u_2|\}$. 
\end{lemma}
\begin{proof} 
Let $\eta'$ and $E_\eta$ be the shadow arc and shadow bigon of $\eta$ respectively. 
By Lemma~\ref{LU12}, each outermost arc has one endpoint in $\partial E_1$ and the other endpoint in $\mathfrak{u}'$. This implies that  $|\Int(\eta')\cap\partial E_1|=k_1$ and hence $|\partial D_\eta'\cap\partial E_1|\le k_1$ after isotopy. 

Let $k$ be the total number of outermost arcs of $D'\cap\Gamma_i$ before the two $\partial$-compressions described above. 
We have two cases to discuss:

\vspace{8pt}
\noindent
Case (i) $D'\cap\Gamma_i$ contains an outermost arc outside the shadow bigon disk $E_\eta$
\vspace{8pt}

In this case, we have 
$k_1<k$. By Lemma~\ref{Lminimality}, we have $|\partial D_\eta'\cap\partial E_1|\le k_1<k\le \min\{|v_j\cap u_1|, |v_j\cap u_2|\}$, and the lemma holds.

\vspace{8pt}
\noindent
Case (ii) all the outermost arc of $D'\cap\Gamma_i$ lie in the shadow bigon disk $E_\eta$, i.e.~$k_1=k$.
\vspace{8pt}

In this case, we view the annulus $\Gamma_i$ as a sub-surface of the meridional disk $V_i$. Since all the outermost arcs lie in the bigon disk $E_\eta$, there must be an arc $\beta$ of $D'\cap\Gamma_i$ that is outermost in $V_i$ but not outermost in $\Gamma_i$, see Figure~\ref{Fbeta}.  In other words, the outermost bigon disk determined by $\beta$ in $V_i$ contains the disk $V_i\setminus\Gamma_i$. 
Note that this means that the shadow bigon of $\beta$ contains   $E_\eta$.
Let $\beta''$ be the arc in $v_i\times\{1\}$ such that $\partial\beta''=\partial\beta$ and $\beta\cup\beta''$ bounds the outermost disk $E_\beta$ in $V_i$, see the shaded region in Figure~\ref{Fbeta} for a picture of $\beta$ and $E_\beta$. Since $\beta$ is outermost in $V_i$,  $\Int(\beta'')\cap D'=\emptyset$ and $V_i$ is the union of $E_\beta$ and the shadow bigon of $\beta$.

\begin{figure}[h]
	\begin{overpic}[width=1.3in]{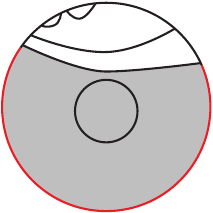}
		\put(52,80){$\eta$}
		\put(76,71){$\beta$}
		\put(98,33){$\beta''$}
		\put(65,23){$E_\beta$}
	\end{overpic}
	\caption{}\label{Fbeta}
\end{figure}

We view $\beta$ as an arc in $D'$ before the two $\partial$-compressions. 
By Lemma~\ref{Lu0} (replacing $D$, $\mathfrak{u}$, $\alpha$, $\alpha'$ and $U_2$ in Lemma~\ref{Lu0} with $D'$, $\mathfrak{u}'$, $\beta$, $\beta''$ and $U_2'$ respectively),  if $\beta$ has both endpoints in $\mathfrak{u}'$, then $\Int(\beta'')\cap D'\ne\emptyset$, contradicting our construction of $\beta''$. So,
$\beta$ cannot have both endpoints in $\mathfrak{u}'\subset\partial D'$. 
By Lemma~\ref{LinnerU} (replacing $D$, $\Delta_s$, $\alpha$, $\alpha'$ in Lemma~\ref{LinnerU} with $D'$, $E_1$, $\beta$, $\beta''$ respectively), if $\beta$ has both endpoints in $\partial E_1$, then the subdisk of $D'\setminus \Int(E_1)$ bounded by $\beta$ and a subarc of $\partial E_1$ must contain a $U'$-disk, contradicting our assumption that $E_1$ is the only $U'$-disk in $D'$.
Hence, $\beta$ cannot have both endpoints in $\partial E_1$.
Thus, $\beta$ has one endpoint in $\partial E_1$ and the other endpoint of in $\mathfrak{u}'\subset\partial D'$. 
Similar to the picture of $D_\eta$ in Figure~\ref{Feta}(c), 
there is a quadrilateral subdisk $D_\beta$ of $D'$ bounded by the union of
$\beta$, a subarc of $\partial E_1$, a subarc of $\mathfrak{u}'$, and an outermost arc $\alpha$ of $D'\cap\Gamma_i$. 

Now we perform two $\partial$-compressions. The first is along the shadow bigon disk of the outermost arc $\alpha$ and the second $\partial$-compression is along the disk $E_\beta$.  Similar to the disk $D_\eta'$ described earlier, the two $\partial$-compressions push $D_\beta$ into a disk $D_\beta'$ with $\partial D_\beta'\subset S\times\{1\}$.  
Since $E_\beta$ is the complement of the shadow bigon disk of $\beta$ in $V_i$, the second $\partial$-compression pushes $D_\beta$ towards the negative direction, whereas 
the first $\partial$-compression pushes $D_\beta$ towards the positive direction. 
Since $\beta$ has exactly one endpoint in $\partial E_1$, 
this implies that, after isotopy, $\partial E_1$ intersects $\partial D_\beta'$ in a single point, which means that the Heegaard splitting is stabilized and not of minimal-genus, a contradiction. 
This means that Case (ii) cannot happen, and the lemma holds.
\end{proof}

Recall that the Heegaard diagram formed by $\{u_1,u_2\}$ and $\{v_1, v_2\}$ is assumed to have minimal complexity, see Definition~\ref{Dmin}. 
In Definition~\ref{Dmin}, we have a pair of special meridians $v_s$ and $u_t$ ($t=1$ or $2$ and $s=1$ or $2$) such that $|v_s\cap u_t|$ is minimal among all essential disk pairs in $\mathcal{V}$ and $\mathcal{U}$. 
Consider the special meridian $v_s$ and the annulus $\Gamma_s=v_s\times I$. 
If $D'\cap\Gamma_s$ contains an almost outermost arc $\eta$, Lemma~\ref{Lalmost} implies that $|\partial D_\eta'\cap\partial E_1|<|v_s\cap u_t|$.  
Since $\partial E_1$ is a meridian of $\mathcal{U}$ and $\partial D_\eta'$ is a meridian of $\mathcal{V}$,  this contradicts the hypothesis that the Heegaard diagram is minimal.
Thus, $D'\cap\Gamma_s$ does not contain any almost outermost arc.  Hence, every arc of $D'\cap\Gamma_s$ is outermost.

Since $\Gamma_1\supset\mathfrak{v}$, $D'\cap\Gamma_1$ always contains an almost outermost arc. This means that $v_2$ must be the special meridian in the complexity, i.e.~$s=2$, and hence every arc of $D'\cap\Gamma_2$ is outermost in $\Gamma_2$. 

Let $k$ be the number of arcs in $D'\cap\Gamma_2$. Since every arc of $D'\cap\Gamma_2$ is outermost, Lemma~\ref{Lminimality} implies that $k\le \min\{|v_2\cap u_1|, |v_2\cap u_2|\}$. 
By Lemma~\ref{LU12}, each outermost arc has one endpoint in $\partial E_1$ and the other endpoint in $\mathfrak{u}'\subset\delta'$.  This implies that  $k=|v_2\cap\partial E_1|$. 
Hence, $|v_2\cap\partial E_1|\le\min\{|v_2\cap u_1|, |v_2\cap u_2|\}$.
We concluded earlier that $v_2$ is a special meridian in Definition~\ref{Dmin}. Let $u_t$ ($t=1$ or $2$) be the other special meridian.  As the complexity of the Heegaard diagram is minimal, the equality must hold, that is, $\min\{|v_2\cap u_1|, |v_2\cap u_2|\}=|v_2\cap u_t|=|v_2\cap\partial E_1|$.

As $D'\cap\Gamma_1$ always contains an almost outermost arc, let $\eta$ be an almost outermost arc in $\Gamma_1$ and let $D_\eta'$ be the disk in Lemma~\ref{Lalmost}.  
Recall that the disk $D_\eta'$ is obtained by two $\partial$-compressions and the first $\partial$-compression is along all outermost bigons. 
Since every arc of $D'\cap\Gamma_2$ is outermost, $D'$ becomes disjoint from $\Gamma_2$ after the first $\partial$-compression, i.e.~$D''\cap\Gamma_2=\emptyset$, see Lemma~\ref{Lparallel}.  This means that $\partial D_\eta'$ is disjoint from $v_2$.

We have concluded before Lemma~\ref{Lalmost} that $\partial D_\eta'$ is a meridian of $\mathcal{V}$. 
By Lemma~\ref{Lalmost}, $|\partial D_\eta'\cap\partial E_1|<\min\{|v_1\cap u_1|, |v_1\cap u_2|\}$.
We have two cases to discuss:

The first case is that $\partial D_\eta'$ is isotopic to $v_2$.  In this case, we have $|v_2\cap\partial E_1|=|\partial D_\eta'\cap\partial E_1|<\min\{|v_1\cap u_1|, |v_1\cap u_2|\}$. However, we have concluded earlier that $|v_2\cap u_t|=|v_2\cap\partial E_1|$, where $v_2$ and $u_t$ are the special meridians in Definition~\ref{Dmin}.  This means that 
$|v_2\cap u_t|<\min\{|v_1\cap u_1|, |v_1\cap u_2|\}$ and this contradicts our definition of special meridians, see Definition~\ref{Dmin}.

It remains to consider the case that $\partial D_\eta'$ is not isotopic to $v_2$. 
We have concluded earlier that $\partial D_\eta'$ is disjoint from $v_2$. 
So, in this case,  $\{v_2, \partial D_\eta'\}$ is a complete set of meridians for $\mathcal{V}$. 
By adding a meridian $u_E$ of $\mathcal{U}$ that is disjoint from $\partial E_1$, we can construct a new complete set of meridians $\{\partial E_1, u_E\}$ of $\mathcal{U}$. 
By the conclusions $|\partial D_\eta'\cap\partial E_1|<\min\{|v_1\cap u_1|, |v_1\cap u_2|\}$ and $|v_2\cap\partial E_1|\le \min\{|v_2\cap u_1|, |v_2\cap u_2|\}$, we have $|(v_2\cup\partial D_\eta')\cap\partial E_1|< \min\{|(v_1\cup v_2)\cap u_1|, |(v_1\cup v_2)\cap u_2|\}$. 
Hence, $|(v_2\cup\partial D_\eta')\cap\partial E_1|< |(v_1\cup v_2)\cap u_t|$. 
These inequalities imply that 
 the new Heegaard diagram formed by $\{\partial E_1, u_E\}$ and $\{v_2, \partial D_\eta'\}$ has smaller complexity after  setting $\partial E_1$ and $v_2$ as the special meridians.   This contradicts the assumption the original Heegaard diagram has minimal complexity.  

Therefore, Proposition~\ref{Pbigon} holds.

\end{document}